\documentclass[11pt]{article}
\usepackage[dvips]{graphicx}
\usepackage{comment}
\usepackage{textcomp}
\usepackage{amsbsy}
\usepackage{latexsym}
\usepackage[mathscr]{eucal}
\usepackage{amsfonts,amsmath,amsthm}
\usepackage{amssymb}
\usepackage[usenames]{xcolor}

\usepackage{fullpage}

\def\rnew{\color{magenta}}

\begin{document}
\bibliographystyle{plain}
\title
{Constrained Kolmogorov widths}
\author{Ronald DeVore, Guergana Petrova, Jonathan W. Siegel, and  Przemys{\l}aw Wojtaszczyk}
%
%
%
\hbadness=10000
\vbadness=10000
\newtheorem{lemma}{Lemma}[section]
\newtheorem{prop}[lemma]{Proposition}
\newtheorem{cor}[lemma]{Corollary}
\newtheorem{theorem}[lemma]{Theorem}
\newtheorem{remark}[lemma]{Remark}
\newtheorem{example}[lemma]{Example}
\newtheorem{definition}[lemma]{Definition}
\newtheorem{proper}[lemma]{Properties}
\newtheorem{assumption}[lemma]{Assumption}
%
\newenvironment{disarray}{\everymath{\displaystyle\everymath{}}\array}{\endarray}

\def\RR{\rm \hbox{I\kern-.2em\hbox{R}}}
\def\NN{\rm \hbox{I\kern-.2em\hbox{N}}}
\def\ZZ{\rm {{\rm Z}\kern-.28em{\rm Z}}}
\def\CC{\rm \hbox{C\kern -.5em {\raise .32ex \hbox{$\scriptscriptstyle
|$}}\kern
-.22em{\raise .6ex \hbox{$\scriptscriptstyle |$}}\kern .4em}}
\def\vp{\varphi}
\def\<{\langle}
\def\>{\rangle}
\def\t{\tilde}
\def\i{\infty}
\def\e{\varepsilon}
\def\sm{\setminus}
\def\nl{\newline}
\def\o{\overline}
\def\wt{\widetilde}
\def\wh{\widehat}
\def\cT{{\cal T}}
\def\cA{{\cal A}}
\def\cI{{\cal I}}
\def\cV{{\cal V}}
\def\cB{{\cal B}}
\def\cF{{\cal F}}
\def\cY{{\cal Y}}

\def\cD{{\cal D}}
\def\cP{{\cal P}}
\def\cJ{{\cal J}}
\def\cM{{\cal M}}
\def\cO{{\cal O}}
\def\Chi{\raise .3ex
\hbox{\large $\chi$}} \def\vp{\varphi}
\def\lsima{\hbox{\kern -.6em\raisebox{-1ex}{$~\stackrel{\textstyle<}{\sim}~$}}\kern -.4em}
\def\lsim{\hbox{\kern -.2em\raisebox{-1ex}{$~\stackrel{\textstyle<}{\sim}~$}}\kern -.2em}
\def\[{\Bigl [}
\def\]{\Bigr ]}
\def\({\Bigl (}
\def\){\Bigr )}
\def\[{\Bigl [}
\def\]{\Bigr ]}
\def\({\Bigl (}
\def\){\Bigr )}
\def\L{\pounds}
\def\pr{{\rm Prob}}
\newcommand{\cs}[1]{{\color{magenta}{#1}}}
\def\ds{\displaystyle}
\def\ev#1{\vec{#1}}     
\newcommand{\lt}{\ell^{2}(\nabla)}
\def\Supp#1{{\rm supp\,}{#1}}
\def\R{\mathbb{R}}
\def\E{\mathbb{E}}
\def\nl{\newline}
\def\T{{\relax\ifmmode I\!\!\hspace{-1pt}T\else$I\!\!\hspace{-1pt}T$\fi}}
\def\N{\mathbb{N}}
\def\Z{\mathbb{Z}}
\def\N{\mathbb{N}}
\def\Zd{\Z^d}
\def\Q{\mathbb{Q}}
\def\C{\mathbb{C}}
\def\Rd{\R^d}
\def\gsim{\mathrel{\raisebox{-4pt}{$\stackrel{\textstyle>}{\sim}$}}}
\def\sime{\raisebox{0ex}{$~\stackrel{\textstyle\sim}{=}~$}}
\def\lsim{\raisebox{-1ex}{$~\stackrel{\textstyle<}{\sim}~$}}
\def\div{\mbox{ div }}
\def\M{M}  \def\NN{N}                  
\def\Le{{\ell^1}}            
\def\Lz{{\ell^2}}
\def\Let{{\tilde\ell^1}}     
\def\Lzt{{\tilde\ell^2}}
\def\Ltw{\ell^\tau^w(\nabla)}
\def\t#1{\tilde{#1}}
\def\la{\lambda}
\def\La{\Lambda}
\def\ga{\gamma}
\def\BV{{\rm BV}}
\def\Ga{\eta}
\def\al{\alpha}
\def\cZ{{\cal Z}}
\def\cA{{\cal A}}
\def\cU{{\cal U}}
\def\argmin{\mathop{\rm argmin}}
\def\argmax{\mathop{\rm argmax}}
\def\prob{\mathop{\rm prob}}

\def\cO{{\cal O}}
\def\cA{{\cal A}}
\def\cC{{\cal C}}
\def\cS{{\cal F}}
\def\bu{{\bf u}}
\def\bz{{\bf z}}
\def\bZ{{\bf Z}}
\def\bI{{\bf I}}
\def\cE{{\cal E}}
\def\cD{{\cal D}}
\def\cG{{\cal G}}
\def\cI{{\cal I}}
\def\cJ{{\cal J}}
\def\cM{{\cal M}}
\def\cN{{\cal N}}
\def\cT{{\cal T}}
\def\cU{{\cal U}}
\def\cV{{\cal V}}
\def\cW{{\cal W}}
\def\cL{{\cal L}}
\def\cB{{\cal B}}
\def\cG{{\cal G}}
\def\cK{{\cal K}}
\def\cX{{\cal X}}
\def\cS{{\cal S}}
\def\cP{{\cal P}}
\def\cQ{{\cal Q}}
\def\cR{{\cal R}}
\def\cU{{\cal U}}
\def\bL{{\bf L}}
\def\bl{{\bf l}}
\def\bK{{\bf K}}
\def\bC{{\bf C}}
\def\X{X\in\{L,R\}}
\def\ph{{\varphi}}
\def\D{{\Delta}}
\def\H{{\cal H}}
\def\bM{{\bf M}}
\def\bx{{\bf x}}
\def\bj{{\bf j}}
\def\bG{{\bf G}}
\def\bP{{\bf P}}
\def\bW{{\bf W}}
\def\bT{{\bf T}}
\def\bV{{\bf V}}
\def\bv{{\bf v}}
\def\bt{{\bf t}}
\def\bz{{\bf z}}
\def\bw{{\bf w}}
\def \span{{\rm span}}
\def \meas {{\rm meas}}
\def\rhom{{\rho^m}}
\def\diff{\hbox{\tiny $\Delta$}}
\def\EE{{\rm Exp}}
\def\lll{\langle}
\def\argmin{\mathop{\rm argmin}}
\def\codim{\mathop{\rm codim}}
\def\rank{\mathop{\rm rank}}

\def\argmax{\mathop{\rm argmax}}
\def\dJ{\nabla}
\newcommand{\ba}{{\bf a}}
\newcommand{\bb}{{\bf b}}
\newcommand{\bc}{{\bf c}}
\newcommand{\bd}{{\bf d}}
\newcommand{\bs}{{\bf s}}
\newcommand{\bff}{{\bf f}}
\newcommand{\bp}{{\bf p}}
\newcommand{\bg}{{\bf g}}
\newcommand{\by}{{\bf y}}
\newcommand{\br}{{\bf r}}
\newcommand{\be}{\begin{equation}}
\newcommand{\ee}{\end{equation}}
\newcommand{\bea}{$$ \begin{array}{lll}}
\newcommand{\eea}{\end{array} $$}
\def \Vol{\mathop{\rm  Vol}}
\def \mes{\mathop{\rm mes}}
\def \Prob{\mathop{\rm  Prob}}
\def \exp{\mathop{\rm    exp}}
\def \sign{\mathop{\rm   sign}}
\def \sp{\mathop{\rm   span}}
\def \rad{\mathop{\rm   rad}}
\def \vphi{{\varphi}}
\def \csp{\overline \mathop{\rm   span}}

\def\beginproof{\noindent{\bf Proof:}~ }
\def\endproof{\hfill\rule{1.5mm}{1.5mm}\\[2mm]}

\newenvironment{Proof}{\noindent{\bf Proof:}\quad}{\endproof}

\renewcommand{\theequation}{\thesection.\arabic{equation}}
\renewcommand{\thefigure}{\thesection.\arabic{figure}}

\makeatletter
\@addtoreset{equation}{section}
\makeatother

\newcommand\abs[1]{\left|#1\right|}
\newcommand\clos{\mathop{\rm clos}\nolimits}
\newcommand\trunc{\mathop{\rm trunc}\nolimits}
\renewcommand\d{d}
\newcommand\dd{d}
\newcommand\diag{\mathop{\rm diag}}
\newcommand\dist{\mathop{\rm dist}}
\newcommand\diam{\mathop{\rm diam}}
\newcommand\cond{\mathop{\rm cond}\nolimits}
\newcommand\eref[1]{{\rm (\ref{#1})}}
\newcommand{\iref}[1]{{\rm (\ref{#1})}}
\newcommand\Hnorm[1]{\norm{#1}_{H^s([0,1])}}
\def\int{\intop\limits}
\renewcommand\labelenumi{(\roman{enumi})}
\newcommand\lnorm[1]{\norm{#1}_{\ell^2(\Z)}}
\newcommand\Lnorm[1]{\norm{#1}_{L_2([0,1])}}
\newcommand\LR{{L_2(\R)}}
\newcommand\LRnorm[1]{\norm{#1}_\LR}
\newcommand\Matrix[2]{\hphantom{#1}_#2#1}
\newcommand\norm[1]{\left\|#1\right\|}
\newcommand\ogauss[1]{\left\lceil#1\right\rceil}
\newcommand{\QED}{\hfill
\raisebox{-2pt}{\rule{5.6pt}{8pt}\rule{4pt}{0pt}}%
  \smallskip\par}
\newcommand\Rscalar[1]{\scalar{#1}_\R}
\newcommand\scalar[1]{\left(#1\right)}
\newcommand\Scalar[1]{\scalar{#1}_{[0,1]}}
\newcommand\Span{\mathop{\rm span}}
\newcommand\supp{\mathop{\rm supp}}
\newcommand\ugauss[1]{\left\lfloor#1\right\rfloor}
\newcommand\with{\, : \,}
\newcommand\Null{{\bf 0}}
\newcommand\bA{{\bf A}}
\newcommand\bB{{\bf B}}
\newcommand\bR{{\bf R}}
\newcommand\bD{{\bf D}}
\newcommand\bE{{\bf E}}
\newcommand\bF{{\bf F}}
\newcommand\bH{{\bf H}}
\newcommand\bU{{\bf U}}
\newcommand \A {{\bb A}}
\newcommand\cH{{\cal H}}
\newcommand\sinc{{\rm sinc}}
\def\enorm#1{| \! | \! | #1 | \! | \! |}

\newcommand{\am}{a_{\min}}
\newcommand{\aM}{a_{\max}}

\newcommand{\dm}{\frac{d-1}{d}}

\let\bm\bf
\newcommand{\bbeta}{{\mbox{\boldmath$\beta$}}}
\newcommand{\bal}{{\mbox{\boldmath$\alpha$}}}
\newcommand{\bbi}{{\bm i}}

\def\nnew{\color{Red}}
\def\mnew{\color{Blue}}
\def\wnew{\color{magenta}}
\def\gnew{\color{green}}

\newcommand{\dI}{\Delta}
\newcommand\aconv{\mathop{\rm absconv}}

\maketitle
\begin{abstract}
The main theme of approximation theory is to understand how well a general function $f$
can be approximated by a simpler function $g$ such as a polynomial or spline.  In many applications, one wants  $g$ to retain known properties of $f$ such as its inherent smoothness or a geometrical property such as monotonicity or convexity.  Additional requirements on $g$ of this type are known as constraints.  In this paper, we do a systematic study of constrained approximation to understand how the imposition of such constraints limits the efficiency of the approximation.   We study constrained approximation in the setting of linear approximation where $g$ is to be taken from a finite dimensional linear space $V$ of a fixed dimension $n$. Kolmogorov widths describe how well one can approximate when using such linear spaces $V$.  The first part of this
paper introduces and studies several types of constrained widths, including the constrained Kolmogorov widths,  and gives comparisons between
them.  The second part of the paper is restricted to  classical settings where the constraint imposes a smoothness requirement on $g$.   In this case,   our results prove that the additional constraint can typically be imposed with no loss in the efficiency of the approximation.
\end{abstract}

\section{Introduction}
\label{S:elementary}
Kolmogorov widths are a measure of how well we can approximate the elements of a compact set $K$ when using linear spaces $V$ of a fixed dimension $n\ge 0$.  There are numerous generalizations of these  widths to other settings of approximation (see e.g. \cite{Pinkus,T}).   Two prominent   examples are linear widths and Gelfand widths.  We also  have  the notions  of manifold widths \cite{DHM}, stable manifold widths \cite{CDPW}, and Lipschitz widths \cite{PW},  which replace the linear spaces $V$ by more general sets.  The theory of $n$ widths is important in application domains such as numerical analysis and learning theory since they provide a benchmark for the optimal performance of numerical algorithms in the corresponding application domain.

The present paper is concerned with another type of width, called the {\it constrained Kolmogorov width}, where the approximation is again required to come from a finite $n$ dimensional linear space $V$, but the approximation is subjected to an additional constraint whose exact form depends on the targeted application domain of the approximation process such as numerical PDEs, learning, geometric design, etc.  

\subsection{ Kolmogorov, Gelfand and linear widths}
Before formally defining the constrained Kolmogorov widths, we first recall the definition of the standard Kolmogorov widths.  Let $X$ be a Banach space with norm $\|\cdot\|_X$.  We denote the collection of all compact subsets of $X$ by
\be 
\label{defcK}
\cK:=\cK(X):=\{K:\ K\subset X \ {\rm is \ compact}\},
\ee  
and by 
$$
\cK_p\subset \cK(X)
$$
the collection of those subsets of $\cK(X)$  that are   convex and centrally symmetric about the origin.
Given a compact set $K\in \cK(X)$, the $n$-th Kolmogorov width of $K$ is given by
\be 
 \label{Kolwidth}
d_0(K)_X:=\sup_{f\in K}\,\,\|f\|_X, \quad d_n(K)_X := \inf_{\dim(V)\leq n} \,\sup_{f\in K}\,\,\inf_{v\in V}\|f-v\|_X, \quad n\in\N,
\ee
where the infimum is taken over all $n$ dimensional linear subspaces $V$ of $X$.  In other words, the Kolmogorov width describes how well the elements in $K$  can be approximated by an optimally chosen $n$ dimensional
space $V\subset X$.    If a subspace $V$ of $X$ with dimension $n$ satisfies
\be 
\label{Kolspace}
\dist(K,V):=\sup_{f\in K}\,\,\inf_{v\in V}\|f-v\|_X=d_n(K)_X,
\ee
we call $V$ a {\it Kolmogorov subspace} for $K$ with respect to  the norm $\|\cdot\|_X$.  Such a space $V$ does not always exist,  but for each $\e>0$ there is always a space $V$ whose distance to $K$ is at most $d_n(K)_X+\e$.  In this case, we say $V$ is an $\e$-Kolmogorov space. 

\begin{remark} 
\label{R:firstremark} When establishing bounds for Kolmogorov widths or the constrained Kolmogorov widths introduced below, it is easier to prove these bounds under the assumption that these widths are attained by an $n$ dimensional space $V$.  In the absence of such knowledge,  one uses spaces $V$ that approximate within $\e>0$ of the width, and then takes a limit as $\e\to 0$.  We want to avoid the constant repetition of this $\e$ argument and so in the proofs given below we assume the existence of a Kolmogorov space.  The reader can easily substitute the $\e$ argument to handle the general case.  When it is not clear how the $\e$ argument works we provide details.
\end{remark}

Two other important widths which we will mention and use later are the {\it linear}  and {\it Gelfand widths}. 
For a compact set $K\subset X$, the linear widths are defined by
\begin{equation}
    \delta_n(K)_X :=\inf_{A_n} \sup_{f\in K}\|f - A_nf\|_X,\quad n\in\N,
\end{equation}
where the infimum is taken over all  linear continuous  operators $A_n:X\rightarrow X$ of rank at most $n$. The linear widths measure how well the elements of $K$ can be approximated by a rank (at most) $n$ linear map.

The Gelfand widths of a   set $K\in\cK_p(X)$ is defined by
\begin{equation}\label{gelfand-widths-definition}
    d^n(K)_X := \inf_{\lambda_1,...,\lambda_n\in X^*}\sup\{\|f\|_X:~f\in K,~\lambda_i(f) = 0, \ i=1,\dots,n\},\quad n\in\N,
\end{equation}
where the infimum is taken over any $n$ continuous linear functionals from the dual space $X^*$. 
The Gelfand widths measure how accurately elements of $K$ can be recovered from $n$ linear measurements $\lambda_1,...,\lambda_n\in X^*$. It is easy  to see  that 
$$
d_n(K)_X\leq \delta_n(K)_X \quad {\rm and} \quad d^n(K)_X\leq \delta_n(K)_X, \quad n\in\N.
$$

A significant chapter of approximation theory has been devoted to determining the Kolmogorov, Gelfand, and linear widths of a specified $K\in \cK$ and specified Banach space $X$, or at least to determine how fast 
these widths tend to zero as $n\to\infty$. The set $K$ is commonly referred to as a {\it model class} since it  serves as an assumption about the functions $f\in X$ that we wish to approximate.  A simple example to keep in mind is 
the case when  $X=C([0,1])$ is   the space of continuous functions on $\Omega=[0,1]$ with the uniform norm $\|\cdot\|_X$, and  the set $K=U({\rm Lip }\ \alpha)$, $0<\alpha\le 1$, is  the unit ball of ${\rm Lip}\ \alpha$. This  is a compact subset of $X$ and its Kolmogorov $n$ widths satisfy $d_n(K)_X\asymp n^{-\alpha}$, $n\ge 1$, with   constants of equivalency independent of both $n$  and $\alpha$.

\subsection{ Constrained Kolmogorov widths}
\label{SS:constrainedwidths}
Some applications may require  more  restrictions on the approximating functions than the requirement  that they  come from  a finite dimensional linear space $V \subset X$.  Let us denote by $\cC\subset X$  a subset of $X$, which we refer to as the {\it constraining class}.  In addition to requiring that the approximating functions come from an $n$ dimensional space $V$, we may also want to require that these approximants  come from the class $\cC$.  This setting is referred to as constrained approximation.  The most common constraints relate to  some geometric property such as positivity, monotonicity, or convexity. 
Another common constraint is simply the requirement that  the approximating function be in the model class $K$, in which case  we have $\cC=K$.  This latter example is
important in learning theory (see \cite{BBDP}).

Formally, we define the {\it constrained Kolmogorov widths} of $K$ to be
 \be 
 \label{Kolwidthc1}
d_0(K,\cC)_X := \sup_{f\in K}\|f\|_X, \quad d_n(K,\cC)_X := \inf_{\dim(V)\leq n} \sup_{f\in K}\,\inf_{v\in V\cap \cC}\|f-v\|_X, \quad n\in\N.
\ee
Note that we obviously have 
\be 
\label{KcK}
d_n(K)_X \leq d_n(K,\cC)_X,
\ee
for any $\cC\subset X$.
The constrained Kolmogorov widths \eqref{Kolwidthc1} were first introduced in \cite{konovalov1984}, where they were called relative widths. It was shown in \cite{konovalov1984} that the constrained Kolmogorov widths    can be substantially different from the standard Kolmogorov widths for certain classical smoothness classes $K$. Different choices for $\cC$ yield different constrained Komogorov widths.

In this paper, we are particularly interested in a special class  of constrained widths that arises in learning theory (see \cite{BBDP}), where one wants the approximant to
be in $K$, or at least be close to being in $K$. Our first example of such a width is 
 
 \be 
 \label{Kolwidthc}
d^{c}_n(K)_X := d_n(K,K)_X:= \inf_{\dim(V)\leq n} \sup_{f\in K}\,\inf_{v\in V\cap K}\|f-v\|_X, \quad n\in\N.
\ee
This width corresponds to taking $\cC=K$.  For example, if $K$ were the unit ball of Lip $1$, we would want the approximating functions to also have this property.

This constrained width is quite demanding and we would like to relax this constraint. To this end, for each real number $\gamma\ge 1$,  we introduce the sets
\be 
\label{inflatedsets}
K_{\gamma}:={\displaystyle\bigcup_{0\le \lambda\leq \gamma}\lambda K}. 
\ee
Note that when $K$ is convex and centrally symmetric, we have $\gamma K=K_\gamma$. We define the $\gamma$-constrained Kolmogorov width
\be 
 \label{Kolwidthnew}
d^{c,\gamma}_n(K)_X := d_n(K,K_\gamma)_X := \inf_{\dim(V)\leq n} \sup_{f\in K}\,\inf_{v\in V\cap K_{\gamma}}\|f-v\|_X, \quad n\in\N.
\ee
When $\gamma=1$, we note that for centrally symmetric convex sets $K \in\cK_p$ we have $K_1=K$, and hence
\be\label{constrw}
d^{c}_n(K)_X =d^{c,1}_n(K)_X, \quad n\geq 1.
\ee
The widths $d_n^c$ and $d_n^{c,\gamma}$ have been studied by numerous authors \cite{babenko1991,konovalov2002,malykhin2016relative,telyakovskii2001}  for particular cases of Banach spaces X.

We next introduce some further  widths that are closely related to the concept of constrained approximation. 

\subsection{Greedy Kolmogorov widths}
\label{SS:greedy}
We define the  {\it greedy Kolmogorov widths} of $K$   as
\be\label{greedywidth}
d_0^g(K)_X:=\sup_{f\in K}\,\|f\|_X, \quad d_n^g(K)_X:=\inf_{V=\span\{f_1,\ldots,f_n\}, f_i\in K} \sup_{f\in K}\,\inf_{v\in V}\|f-v\|_X,\quad n\in\N,
\ee
where the $n$ dimensional spaces $V$ allowed in the competition are spanned by elements of $K$.  This width  has been indirectly studied in the context of the construction of greedy bases which have, in turn,  been used in the study of reduced modeling (see \cite{BMPPT,BCDDPW,DPWg}).  In some settings, it is known that the decay rate of $d_n^g(K)_X$, $n\to\infty$, is comparable to the decay rate for $d_n(K)_X$.  For example, in the case when $X$ is a Hilbert space, if we know that if for some $\alpha>0$, we have  $d_n(K)_X\le Mn^{-\alpha}$, $n\ge 1$,  then it follows that $d^g_n(K)_X\le C(\alpha)Mn^{-\alpha}$, $n\ge 1$  (see \cite{BCDDPW}). 

Notice that the knowledge of $d_n^g(K)_X$ does not provide us with any direct information about $d_n^c(K)_X$ or $d_n^{c,\gamma}(K)_X$ since we do not know that the approximants are in $K_\gamma$.  On the other hand, if 
$K$ is convex and centrally symmetric, then the approximant will be in $K_\gamma $ if we can provide a bound on the coefficients in the approximation.  This approach of bounding coefficients will be studied later in this paper.

\subsection{ Restricted Kolmogorov widths}
\label{SS:restricted}
Note that one can view $K$ not only as a subset of $X$, but also as a subset of other ambient spaces.   
  For example, given a compact set $K$, let $\cX:= \cX(K)$  be  the closure of the finite linear span $S(K)$ of $K$ with respect to  $\|\cdot\|_X$.  We equip $\cX$ with    the norm in $X$. That is,  $\|g\|_{\cX}:=\|g\|_X$  when $g\in \cX$.  Then, $\cX(K)\subset X$ is the smallest subspace of $X$ which contains $K$.  With this notation in hand, we define the {\it restricted Kolmogorov width} of $K$ to be
\be 
\label{restrictedwidth}
 d_0^r(K)_X:=\sup_{f\in K}\,\|f\|_X, \quad d^{r}_n(K)_X:=d_n(K)_{\cX}= \inf_{V\subset \cX, {\rm dim}(V)\leq n} \sup_{f\in K}\,\inf_{v\in V}\|f-v\|_X,\quad  n\in\N.
\ee 
This restricted width has been studied in other contexts (see
\cite{I,O,OO}). Note that 
the restricted width is a constrained width, namely, it can also be defined by
$$
d^{r}_n(K)_X:=d_n(K,\cX(K))_X.
$$

The reader should notice that in both the greedy and constrained widths, the norm in which the error of approximation is measured continues to be $\|\cdot\|_X$.
The added property imposed in these widths is to limit the $n$ dimensional spaces $V$ which can be used to do the approximation.  In greedy widths, these $n$ dimensional spaces are required to be spanned by the elements of $K$, whereas in the restricted widths the spaces $V$ must be spanned by elements of $\cX(K)$.  It follows therefore that 
\be\label{obs1}
d^{r}_n(K)_X\leq d^g_n(K)_X, \quad n\in\N.
\ee
One may expect that these two widths are the same.  
This is indeed the case 
for $K\in\cK_p$, that is 
 for convex, centrally symmetric sets, as  shown in \S\ref{S:rgcwidths}.

\subsection{The goals of this paper}
\label{SS:goals}

By now, we have presented several different types of widths.  
Our main interest in this paper are the $\gamma$-constrained  Kolmogorov  widths $d_n^{c,\gamma}(K)_X$, while the  greedy and restricted Kolmogorov widths,  $d_n^g(K)_X$ and  $d_n^r(K)_X$,   have been  introduced as tools to better understand 
the behavior of $d_n^{c,\gamma}(K)_X$.
The goal of this paper is to investigate the asymptotic behavior of these  constrained widths of a compact set $K$ as $n\to \infty$.  In particular, we want to understand
if and when the addition of the constraint denegrates the rate of approximation.

We obtain results of two flavors.  The first type of results are statements for constrained  or $\gamma$-constrained Kolmogorov widths 
that hold for any compact set $K$ or for any convex centrally symmetric compact set $K$ and any Banach space $X$.  In this generality,  we shall see that there may be a severe loss in the rate of constrained approximation.  Result of this type are described and discussed in   \S2-\S4.

The second set of results we obtain (starting in \S5) apply to specific classical compact sets $K$ and  particular Banach spaces $X$ such as the $L_p$ spaces.  The  compact sets $K$ that 
arise in classical settings of approximation theory studied in this paper are Besov classes and approximation classes.  Our main results for these classical settings will show that $\gamma$-constrained Kolmogorov widths have the same asymptotic rate of decay as the Kolmogorov widths.




\section{Some properties of Kolmogorov widths}
\label{PKol}
We start this section with the observation that the Kolmogorov widths of a set $K$ and its convexification $\widetilde K$ are the same. More precisely, 
for a given a compact set $K\in X$, we denote its symmetrized convex hull by 
\be 
\label{defconvexhull}
\widetilde K:=\overline{{\rm conv}(-K\cup K)},
\ee
where the closure is taken with respect to $\|\cdot\|_X$. 
It is easy to see that,  (see Chapter 13, \cite{LGM}),
\be 
\label{convexify}
d_n(K)_X=d_n(\widetilde K)_X,\quad n\ge 0.
\ee 
 Indeed, since $K\subset \widetilde K$ we have that that $d_n(K)_X\leq d_n(\widetilde K)_X$. The reverse inequality 
 holds because whenever a linear space $V$ approximates a set $K$ to a given accuracy then it approximates 
$-K\cup K$ to the same accuracy.  In turn, it approximates the convex hull of $-K\cup K$  and thus  its  closure to the same accuracy. 

Note that an equality of this type  does not generally hold for the $\gamma$-constrained Kolmogorov widths.     However, it is  easy to check that  for any compact set $K$ we have 
\be 
\label{constconvex}
 d_n^{c,\gamma}(\widetilde K)_X \leq d^{c,\gamma}_n(K)_X,\quad  \ n\in\N.
\ee

We next prove the following theorem for the Kolmogorov widths  which we will use later for particular choices of Banach spaces $Y$ and $Z$.   We consider any Banach space $Z$ and any subspace $Y$ of $Z$. 
 Note that this assumption means that $\|\cdot\|_Y$ is the restriction of $\|\cdot\|_Z$ onto $Y$.  
\begin{theorem}
\label{TABS}
 For every Banach space  $Z$,  every subspace  $Y$ of $Z$, and any compact set  $K\in\cK(Y)$, 
 we have that 
\begin{equation}
\label{T:abs}    
d_{n}(K)_Z\leq d_{n}(K)_Y\leq (1+\sqrt{n})d_n(K)_Z,\quad n\ge 0.
\end{equation}
If  $Z$ is a Hilbert space, then
\be 
\label{philbertbound}
d_{n}(K)_Y=  d_n(K)_Z,\quad n\ge 0.
\ee
\end{theorem}
\noindent
{\bf Proof:}   Notice that the difference in these two Kolmogorov  widths is that $d_n(K)_Z$ allows the $n$ dimensional subspaces $V$ to come from the larger Banach space $Z$,  while $d_n(K)_Y$ requires them to come from  the smaller space $Y$. Therefore, the  first inequality in \eref{T:abs} follows trivially

To show the second inequality in \eref{T:abs}, we fix $n$ and  let $\e>0$ be arbitrary. We 
  choose $Z_n\subset Z$, with $\dim Z_n=n$, such that 
$$
\sup_{f\in K} \inf_{g\in Z_n} \|f-g\|_{Z} \le d_n(K)_Z+\varepsilon.
$$
 Let us define 
$$
\tilde Z:= Y+Z_n.
$$
Clearly $Y$ is a subspace of $\tilde Z$ of codimension $\leq n$ . It follows that for each $\delta>0$,  there exists a linear projector $Q_\delta:\tilde Z\to Y$ onto $Y$  such that (see Theorem 8 in \cite{GG}),

\be 
\label{normQ}
\|Q_\delta\|\leq (1+\sqrt n+\delta).
\ee
We now consider the linear space $Y_n:=Q_\delta(Z_n)\subset X$ which has dimension at most $n$. 

 If $f\in K$, let $g(f)\in Z_n$ be such that 
$$
\|f-g(f)\|_Y=\|f-g(f)\|_{Z} \le d_n(K)_Z+\varepsilon.
$$
Then $g^*(f):=Q_\delta(g(f))\in Y_n$, and we have
\be
\label{fXn}
\|f-g^*(f)\|_Y=\|Q_\delta(f-g(f))\|_Y\le \|Q_\delta\|\|f-g(f)\|_Z\le (1+\delta+\sqrt{n})(d_n(K)_Z+\e).
\ee
It follows that
\be 
   d_n(K)_Y \leq (1+\sqrt n+\delta) (d_n(K)_Z+\epsilon).  
\ee
Letting $\epsilon,\delta\to 0$ gives  \eref{T:abs}.
When $Z$ is a Hilbert space the projection from $\tilde Z$ onto $Y$ can be chosen of norm one and therefore \eref{philbertbound} follows.  This completes the proof of the theorem.
\hfill $\Box$

\section{Elementary properties and comparisons between constrained, $\gamma$-constrained, greedy,  and restricted Kolmogorov  widths}
\label{SS:elementary}

In this section,  we will describe some elementary properties of the constrained Kolmogorov widths. We start with recalling the  notion of entropy numbers   $\varepsilon_n(K)_X$ of a compact set $K$.  Given a $g\in X$ and $r>0$, we let $B(g,r):=B(g,r)_X$ denote the closed ball in $X$ of radius $r$ centered at $g$.
For every fixed $n\geq 0$, the entropy number $\varepsilon_n(K)_X$ is the infimum over all 
$\varepsilon>0$ for which $2^n$  balls
with centers from $X$ and radius $\varepsilon$ cover $K$. 
\be
\label{defentropy}
\epsilon_n(K)_X:= \inf\{\epsilon > 0:\, K\subset 
\bigcup_{j=1}^{2^n}B(g_j,\varepsilon), \,g_j \in X, \quad
j = 1, . . . , 2^n\}.
\ee

In some applications, one wants the centers of the covering of $K$ to come from $K$ itself.  This leads to the definition of the inner entropy numbers
\be 
\label{innerentropy}
\tilde \epsilon_n(K)_X:= \inf\{\epsilon > 0:\, K\subset 
\bigcup_{j=1}^{2^n}B(f_j,\varepsilon), \,f_j \in K, \,
j = 1, . . . , 2^n\}.
\ee
An elementary comparison between these to notions of entropy is (see \cite{CS}),  
\be
\label{dompare}
\epsilon_n(K)_X\leq \tilde \varepsilon_n(K)_X\leq 2\varepsilon_n(K)_X.
\ee
The following lemma provides some elementary properties and comparisons between the widths we have introduced.

\begin{lemma}
\label{L1}
If  $X$ is any Banach space and  $K\in\cK(X)$, then  the  following elementary properties hold for  $n=0,1,\dots$.
\vskip .1in
     {\rm (i)} If $ L\subset K\subset \lambda L$, with  $K,L\in \cK_p(X)$, and $\lambda \ge 1$, then $d^{c,\lambda \gamma}_n(L)_X\le d_n^{c,\gamma}(K)_X$ for all $\gamma>0$.
    Secondly, for any $\alpha,\gamma$, we have
 $
d^{c,\alpha \gamma}_n(\alpha K)_X=
\alpha  d_n^{c,\gamma}( K)_X$.
 \vskip .1in
     {\rm (ii)}
   For any $1\leq \gamma\leq \delta$, we have
    $d^{c,\delta}_n(K)_X\leq d^{c,\gamma}_n(K)_X$. 
   \vskip .1in    
       {\rm (iii)} For every $\gamma\geq 1$, we have $d^{c,\gamma}_{n+1}(K)_X\leq  d^{c,\gamma}_n(K)_X$  and $d^{c}_{n+1}(K)_X\leq  d^{c}_n(K)_X$.
       \vskip .1in 
        {\rm (iv)} For any $\gamma\geq 1$, we have $$d_n(K)_X\leq d^{r}_n(K)_X\leq 
        d^g_n(K)_X\leq
        \lim_{\gamma\to\infty}
        d^{c,\gamma}_n(K)_X\leq
        d^{c,\gamma}_n(K)_X\leq d^{c,1}_n(K)_X\leq d^{c}_n(K)_X.$$
       
   \vskip .1in
   {\rm (v)} We have $
   d^c_{2^n}(K)_X\leq 2\epsilon_n(K)_X.
   $
   \vskip .1in
   {\rm (vi)}  We have $ d^c_n(K)_X\to 0$, as $n\to \infty$.
\end{lemma}

\noindent{\bf Proof:} For the proof of the first statement in (i),  let $\delta_n:= d_n^{c,\gamma}(K)_X$. We assume that  there is an  $n$-dimensional  space $X_n$ that 
realizes $\delta_n$; otherwise one should use the $\e$-argument of Remark \ref{R:firstremark}. If $f\in L$, then $f\in K$ and so there is a $g\in X_n$ such that $\|f-g\|_X\le \delta_n$ and $g\in \gamma K$.
It follows that $g\in \lambda \gamma L$ and so $d_n^{c,\lambda\gamma}(L)_X\le \delta_n$ as desired. Similarly,  for the second statement of (i), let $V$ again attain $\delta_n^{c,\gamma}(K)_X$.  If $f\in\alpha K$ then $f=\alpha f_0
$ with $f_0\in K$.  Therefore, there is a $g_0\in V$ that approximates $f_0$ to accuracy $\delta_n$ and $ g_0\in K_\gamma=\gamma K$. Hence, $g:=\alpha g_0$ is in $V$ and approximates  $f$ to the accuracy $\alpha \delta_n$.  Moreover, $g\in \alpha\gamma  K$.  This gives the inequality $d_n^{c, \alpha\gamma }(\alpha K)_X \le \alpha d_n^{c,\gamma}(K)$.
If we reverse the roles of $K$ and $\alpha K$,  we  obtain $d_n^{c,\gamma}(K)_X\le \alpha^{-1} d_n^{c,\alpha\gamma}(\alpha K)_X$, as desired.

The monotonicity claimed in (ii)  follows from the fact that   $K_\gamma\subset K_{\delta}$, when $\gamma\le\delta$. Claim (iii) is obvious. 

We now turn to the proof of (iv). The first inequality in (iv) follows from the definition of the restricted width. The next inequality was given in \eref{obs1}.  
We next  prove that
\be 
\label{firstshow}
d_n^g(K)_X\le d_n^{c,\gamma}(K)_X, \quad \gamma\ge 1,
\ee 
To prove this, we fix $\gamma\geq 1$ and note that 
for any $n$ dimensional subspace $V\subset X$, we have 
$$
V\cap K_\gamma\subset \span\{f_1, \ldots,f_n\},
$$
for some $f_1, \ldots,f_n\in K$, and therefore
\be
d_n^g(K)_X\le \sup_{f\in K}\inf_{g\in V\cap K_\gamma}\|f-g\|_X.
\ee
Taking an infimum over all $n$ dimensional spaces $V$, we arrive at
$$
d_n^g(K)_X\le d_n^{c,\gamma}(K)_X,
$$
which is \eref{firstshow}.  Next,  note that  $d_n^{c,\gamma}(K)_X$ is a monotone decreasing function of $\gamma$, and so we have 
$$
d_n^g(K)_X\le \lim_{\gamma\to\infty}
        d^{c,\gamma}_n(K)_X,
$$
which is the third inequality in (iv).  The remainder of inequalities in (iv) are obvious.

To prove (v), 
let us fix $n$ and take  any  $\varepsilon>\tilde \varepsilon_n$.    Let   $f_j\in K$, $j=1,\ldots,2^n$, be the centers of  balls $B(f_j,\varepsilon)$ that cover $K$. Clearly, if $V_{2^n}:=\span\{f_j\}_{j=1}^{2^n}$, then every $f\in K$ can be approximated by at least one of the $f_j$'s up to accuracy $\varepsilon$. Therefore, we have
$d^c_{2^n}(K)_X\leq \varepsilon$.
Since $\epsilon $ can be taken arbitrarily close to $ \tilde\e_n$, using \eref{dompare}, we have
 
$$
d^c_{2^n}(K)_X\leq \tilde \e_n(K)_X\leq 2\e_n(K)_X.
$$

To prove (vi), we note that
 for compact sets $K$ we have $\e_n(K)_X\to 0$ as $n\to\infty$.  Therefore, from (v),  we  know that $d^c_{2^n}(K)_X\to 0$ as 
$n\to\infty$. The statement then follows from the monotonicity of the constrained Kolmogorov widths as stated in (iii).
\hfill $\Box$

 \bigskip

We next want to show  that part (v) of Lemma \ref{L1} cannot be improved.
Let us denote  by $\ell_p$, $ 0<p\le \infty$,
the infinite dimensional $\ell_p$ space with its canonical basis $e_i$, $i=1,2,\dots$,  where each $e_i$ has coordinate one in the $i$-th position and coordinate zero in all other positions.
 We let $\ell_p(\R^m)$ denote the Banach space $\R^m$ equipped with the $\ell_p$ (quasi-)norm when $0<p\le\infty$.

\begin{lemma}
    \label{Lopt}
For the  Hilbert space $X=\ell_2$,  there is a compact subset $K\subset X$  such that
\be 
\label{entropyclaim}
d^c_{2^n}(K)_{\ell_2}\ge d_{2^n}(K)_{\ell_2}\ge \frac{1}{2\sqrt{2}}\e_n(K)_{\ell_2},\quad n\ge 2.
\ee
\end{lemma}
\noindent
{\bf Proof:}   
To construct such a set $K$, we consider the sequence $(a_j)_{j\ge 1} $, where $a_1:=a_2:=1$ and 
\be
 a_{2^{k}+1} =\cdots=a_{2^{k+1}}:=\frac{1}{k}, \quad k\ge 1.
 \ee
 We take  the  canonical basis $e_1,e_2,\dots,$ for $\ell_2$ and  define the set 
 $$
  K:=\{0,a_1e_1,a_2e_2,\dots\}\subset \ell_2.
$$
Since $a_k\to 0$, as $k\to\infty$, the set $K$ is compact in $\ell_2$.

Next, we prove that
\be 
\label{entabove}
\varepsilon_{n+1}(K)_{\ell_2}\le  \frac{1}{n},\quad n\ge 1.
\ee
Consider the  balls of radius $\frac{1}{n}$ about the points 
\be
\label{firstballs}
0,a_1e_1,a_2e_2,\dots,   a_{2^n}e_{2^n}.
\ee
   There are $2^n+1\le 2^{n+1}$ such balls. We now show that  they cover $K$. Certainly, they cover the points listed in \eref{firstballs}.   Consider any point $a_ke_k$ with $k\ge 2^n + 1$.  It has norm at most $\frac{1}{n}$  and therefore  is in the ball $B(0,\frac{1}{n})$.  This proves that these balls cover $K$
   and  completes the proof of \eref{entabove}.

Finally, we can bound $d_{2^n}(K)_X$ from below as follows.  Let $\widetilde K$ be the closure (in $\ell_2$) of the convex hull of 
$-K\cup K$. From  (iv) in Lemma \ref{L1} and \eref{convexify}, we have 
$$ 
d_N^c(K)_{\ell_2}\ge d_N( K)_{\ell_2}=d_N(\widetilde K)_{\ell_2},\quad N\ge 1.
$$
We next bound $d_N(\widetilde K)_{\ell_2}$ from below.  For this, we recall (see Theorem 3.3 on page 411 of \cite{LGM}) that for the unit ball $U(\ell_1(\R^m)):=\{x\in \mathbb{R}^m,~\|x\|_{\ell_1(\R^m)}\leq 1\}$ of $\ell_1(\R^m)$  in $\ell_2(\R^m)$, we have
\be 
\label{lb11}
d_N(U(\ell_1(\R^m)))_{\ell_2(\R^m)} =\sqrt{1-\frac{N}{m}},\quad 0\le N\le m.
\ee 
Given a value of $m$,  we let $\tilde e_1,\dots,\tilde e_m$ be the vectors in $\R^m$ formed by the first $m$ coordinates of $e_1,\dots, e_m$, that is, the vectors $\tilde e_1,\dots, \tilde e_m$ are the standard basis of $\R^m$.   The convex hull $\Lambda_m$ of ${\pm}a_m\tilde e_1,\dots,{\pm}a_m\tilde e_m$ is $a_m U(\ell_1(\R^m))$. Since $(a_j)_{j\ge 1}$ is a nonincreasing sequence, we have
\be 
\label{obviously}
d_N(\widetilde K)_{\ell_2}\ge a_m d_N(U(\ell_1(\R^m))_{\ell_2(\R^m)}  =a_m\sqrt{1-\frac{N}{m}},\quad 1\le N\le m.
\ee
Hence, taking $N=2^n$ and $m=2^{n+1}$, and using \eref{entabove}, we have  
\begin{eqnarray}
d_{2^n}( K)_{\ell_2} &=& d_{2^n}(\widetilde K)_{\ell_2}\ge  
\frac{1}{\sqrt{2}}a_{2^{n+1}}
= \frac{1}{\sqrt{2}}\cdot \frac{1}{n} \ge \frac{1}{2\sqrt{2}(n-1)}\geq \frac{1}{2\sqrt{2}}\e_n(K)_{\ell_2},\quad n\ge 2,
\end{eqnarray}
which  proves  the statement of the lemma.
\hfill $\Box$

The bound (v) of Lemma \ref{L1}  is very weak but it does hold for a general Banach space $X$ and any compact set $K\in\cK(X)$.  The main goal of this paper is aimed at giving much stronger rates of convergence for constrained widths when more is known about $X$ and $K$.  Before formulating these stronger convergence results, we first establish some fundamental comparisons between the different types of constrained Kolmogorov widths and the standard Kolmogorov widths that hold for general compact sets.

\section{Comparisons between various constrained Kolmogorov   widths}
\label{S:rgcwidths}

In this section, we derive   various  comparisons between the  Kolmogorov width and  the restricted,  greedy, and $\gamma$-constrained Kolmogorov widths.

\subsection{Greedy, restricted, and $\gamma$-constrained  Kolmogorov widths.}
     We first prove the following theorem, which shows that the restricted Kolmogorov widths and the greedy Kolmogorov widths are the same, provided $K\in\cK_p$.
\begin{theorem}
\label{T:gr}
 Let $X$ be any Banach space. If $K\in \cK_{p}=\cK_p(X)$,  then for any $n\in\N$, we have
\be\label{uno}
d_n^r(K)_X=d^g_n(K)_X=\lim_{\gamma\to\infty}d^{c,\gamma}_n(K)_X.
\ee
\end{theorem}
\noindent
{\bf Proof:} We fix the value of $n$. 
Lemma \ref{L1}, part (iv), gives that  for every $K\in \cK(X)$, 
\begin{equation}
    d_n^r(K)_X\le d_n^g(K)_X\leq L:=\lim_{\gamma\to\infty} d_n^{c,\gamma}(K)_X,\quad n\in\N.
\end{equation} 
We will now show that 
 for $K\in \cK_p$, we have
\be 
\label{toshow1}
L\leq d_n^r(K)_X,\quad n\in\N,
\ee
which will complete the proof of the theorem.

 To prove \eref{toshow1}, it is enough to show that for any $0<\epsilon \le 1$,  there exists a $\gamma\geq 1$,  which may depend on $\epsilon$, such that
\begin{equation}
\label{toshow2} 
 L\leq    d_n^{c,\gamma}(K)_X \leq d_n^r(K)_X + \epsilon.
\end{equation}
 Note that the first part of the inequality, i.e. 
 $L\leq    d_n^{c,\gamma}(K)_X$, holds for every $\gamma\geq 1$ since $d_n^{c,\gamma}(K)_X$ is a decreasing function of $\gamma$. 

Let $0<\epsilon \le 1$ be arbitrary and fixed,  and recall that $S(K)$ is the linear span of $K$. From the definition of the restricted width,  there exists a subspace $W_n=W_n(\epsilon):= \span\{w_1,...,w_n\}$ with $w_i\in \cX:=\cX(K) = \overline{S(K)}$,
such that
\begin{equation}
\sup_{f\in K}\,\inf_{v\in W_n}\|f-v\|_X\leq d_n^r(K)_X + \epsilon/2.
\end{equation}
Without loss of generality, we can assume that the $w_i$'s, $i=1,\dots,n$, are a basis for $W_n$.  Thus, each $g\in W_n$ has a unique representation
\be 
\label{uniquerep}
g=\sum_{i=1}^n a_i(g)w_i.
\ee
From the equivalence of norms on a finite dimensional space, we have
\be 
\label{eqnorms}
\sum_{i=1}^n|a_i(g)| \le C_1\|g\|_X,\quad g\in W_n,
\ee
where the constant $C_1$ depends on $W_n$.

For each $f\in K$, let $ S(f)\in W_n$ be a best approximation to $f$ from $W_n$, that is
\begin{equation}
    \left\|f - \sum_{i=1}^n a_i(S(f))w_i\right\|_X \leq d_n^r(K)_X + \epsilon/2.
\end{equation}
It follows that
\be 
\label{follows1}
\|\sum_{i=1}^na_i(S(f))w_i\|_X\le \|f\|_X+d_n^r(K)_X + \epsilon/2\le C_2 ,\quad f\in K,
\ee 
where $C_2$ is a fixed constant,   depending only on $K$. 
Going further, we have
\be 
\label{further}
\sum_{i=1}^n|a_i(S(f))|\le 
C_1\|\sum_{i=1}^na_i(S(f))w_i\|_X
\le C,\quad f\in K,\quad C:=C_1C_2.
\ee

Let $\delta > 0$. Since each $w_i\in \overline{S(K)}$ we can choose a $\widetilde{w}_i\in S(K)$ such that $\|w_i - \widetilde{w}_i\|_X \leq \delta$. Considering the perturbed space $V_n := \span\{\widetilde{w}_1,...,\widetilde{w}_n\}$, we then get for any $f\in K$,
\begin{eqnarray}
    \left\|f - \sum_{i=1}^n a_i(S(f))\widetilde{w}_i\right\|_X  
    &\leq&
   \left\|f - \sum_{i=1}^n a_i(S(f)){w}_i\right\|_X  +\sum_{i=1}^n |a_i(S(f))|\|w_i-\widetilde{w}_i\|_X \nonumber \\ 
   &\leq&
   \left\|f - \sum_{i=1}^n a_i(S(f)){w}_i\right\|_X  +\delta\sum_{i=1}^n |a_i(S(f))| \nonumber \\
   &\leq& d_n^r(K)_X + \epsilon/2 +C\delta.
\end{eqnarray}
 If $0<\delta<\frac{\epsilon}{2C}$, we obtain for all $f\in K$
\begin{equation}
\label{eq12}
    \left\|f - \sum_{i=1}^n a_i(S(f))\widetilde{w}_i\right\|_X
    \leq d_n^r(K)_X + \epsilon.
\end{equation}

We will now show that there is a value of $\gamma$ such that for each $f\in K$, the  approximant in \eref{eq12} is in $K_\gamma$. For this, we recall that each $\widetilde{w}_i\in S(K)$, and since $K\in\cK_p$,  there exists  $\widetilde\gamma>0$ such that $\widetilde{w}_i\in K_{\widetilde\gamma}$, i.e. 
$\widetilde{w}_i=\widetilde \gamma f_i$, for some $f_i\in K$, $i=1, \ldots,n$.  Using \eref{further} and the fact that $K\in\cK_p$, we find that
 for all $f\in K$
\begin{equation}
    \sum_{i=1}^n a_i(S(f))\widetilde{w}_i =
    \widetilde \gamma\sum_{i=1}^n a_i(S(f)) f_i\in K_\gamma,~\text{for}~\gamma= C\widetilde\gamma,
\end{equation}
 where $C$ is the constant in \eref{further}.
It then follows from \eref{eq12} that 
$$
\sup_{f\in K}\,\inf_{v\in V_n\cap K_\gamma}\|f-v\|_X\leq d_n^r(K)_X + \epsilon.
$$
Therefore, we have proven that
\begin{equation}
    d_n^{c,\gamma}(K)_X \leq d_n^r(K)_X + \epsilon,
\end{equation}
which shows \eref{toshow2} and completes the proof.
\hfill $\Box$

\subsection{Comparisons between restricted, greedy,  and standard Kolmogorov widths}
\label{SS:restwidth}

It is clear that the standard Kolmogorov width is no larger  than the other widths we have introduced.  We are interested in how much larger these constraining widths are.   We begin with the case of 
restricted Kolmogorov width since
this is the least demanding of the constrained widths we have introduce.
The following theorem gives an upper bound on  the restricted Kolmogorov width of a compact set $K$ by its  Kolmogorov width.
\begin{theorem}
\label{T:firstcomparison}
For every Banach space $X$ and every $K\in\cK( X)$ we have that
\be
\label{opt1}
d^r_{n}(K)_X
\leq (1+\sqrt{n})d_n(K)_X, \quad {n\geq 0.}
\ee
When $X$ is a Hilbert space, we have
\be 
\label{philbertbound2}
d^r_{n}(K)_X=  d_n(K)_X, \quad n\geq 0.
\ee
\end{theorem}
\noindent
{\bf Proof:} 
 We apply Theorem \ref{TABS} with  the Banach spaces $Y:=\cX(K)\subset X=:Z$.  In the case $X$ is a Hilbert space, we have that   $d^r_{n}(K)_X=  d_n(K)_X$, see \eref{philbertbound}. In the  case of a general Banach space,
 from the definition of restricted widths,  this theorem gives
$$
d^r_{n}(K)_X=d_{n}(K)_{Y}\leq (1+\sqrt{n})d_n(K)_{X},
$$
as desired.  
\hfill $\Box$

In the case of compact sets $K\in\cK_p$,   the estimate  \eref{opt1} was given in Proposition 3.2 from \cite{P} with $d_n^r(K)_X$ substituted by $d^g_{n}(K)_X$. Combining Theorems \ref{T:firstcomparison} and \ref{T:gr} furnishes an alternative proof. Namely, the following holds.
\begin{theorem}
\label{T:known}{\rm [Proposition 3.2 from \cite{P}]}
For every 
 compact subset $K\in \cK_{p}$ of a Banach space $X$ we have
\begin{equation}
\label{T:abs1}    
d^g_{n}(K)_X\leq (1+\sqrt{n})d_n(K)_X, \quad n\ge 0.
\end{equation}
When $X$ is  a Hilbert space $H$, then
\begin{equation}
\label{T:abs3}    
d^g_{n}(K)_H=d_n(K)_H, \quad n\geq 0.
\end{equation}
\end{theorem}
\vskip .1in
\noindent
{\bf Proof:} This follows from Theorem \ref{T:firstcomparison} and the fact that $d_n^g(K)_X=d_n^r(K)_X$  when $K\in \cK_p$, see Theorem \ref{T:gr}.
\hfill $\Box$

Note that   estimate \eref{opt1} holds for any  $K\in\cK(X)$.  In the case when $K\in\cK_p$,  the greedy and restricted widths are identical,  and so we obtain the same bound for greedy widths.  However, for a general compact set $K\in \cK(X)$,  the situation  for greedy widths is different.
Namely, Theorem 3.1 from \cite{P} (see Theorem 4.1 from \cite{BCDDPW} for the result in a Hilbert space) gives 
the following result.

\begin{theorem}{\rm [Theorem  3.1 from \cite{P}]}
    \label{knownT1}
    For any  $K\in\cK(X)$  we have that
    $$
    d^g_n(K)_X\leq (n+1)d_n(K)_X, \quad n=0,1,\ldots.
    $$
     On the other hand, for every $n$ and every $\varepsilon>0$ there is a Hilbert space $ X$ and a compact set $K=K(\varepsilon,n) \in \cK(X)$, such that 
    $$
    d^g_n(K)_X\geq  (n-1-\varepsilon)d_n(K)_X.
    $$
\end{theorem}

\vskip .1in

We next show that the bound in Theorem \ref{T:firstcomparison} is sharp.
 Before doing  this, we gather some  preliminary results.  
We denote by $C(\Omega)$ the space of continuous functions on a compact topological space $\Omega$, equipped with the usual norm
\begin{equation}
    \|f\|_{C(\Omega)} := \sup_{x\in \Omega}|f(x)|,
\end{equation}
and present  the following proposition.

\begin{prop}\label{kolmogorov-gelfand-proposition}
    Let $\Omega$ be any compact topological space and $K$ be any compact, centrally symmetric subset of $C(\Omega)$. Then we have
    \begin{equation}
        d_n(K)_{C(\Omega)} \leq \delta_n(K)_{C(\Omega)} = d^n(K)_{C(\Omega)},
    \end{equation}
where $d^n$ and $\delta_n$ denote the Gelfand and linear widths, respectively.
\end{prop}
Proposition \ref{kolmogorov-gelfand-proposition} essentially follows from classical results concerning the so  called {\it absolute Kolmogorov widths} (see for instance Theorem 1 in \cite{I}). A similar result for the corresponding $s$-numbers of operators can be found in \cite{hinrichs2016carl} (point (ii) above equation 1.3). For the reader's convenience, we give {in the Appendix, see \S\ref{SS:A1}}, a self-contained proof.

 We now prove a result that shows that the bound in Theorem \ref{T:firstcomparison} is sharp.
\begin{theorem}
\label{T:sharp}
There exists a Banach space $X$ and  a  set $K\in\cK(X)$ for which
\be
\label{sharp}
d_n^r(K)_X\ge C\sqrt{n}d_n(K)_X,\quad n\ge 0,
\ee
with $C>0$ independent of $n$.
\end{theorem}
\vskip .1in
\noindent
{\bf Proof:}  
We start with a set  $K\in \cK(X)$ and a separable Banach space $X$,  such that 
\begin{equation}
 \label{to3}
    d_n(K)_X\ge C\sqrt{n}d^n(K)_X.
\end{equation}
Examples of such 
 pairs $K$ and $X$ are well-known. For example, in going further,  we can take  
 $$
 K = U(W^1(L_1([0,1])), \quad   X = L_2([0,1]),
 $$
 (see \cite{Kashin} or Chapter 14 of \cite{LGM}).


We next use the fact (Banach-Mazur theorem) that  every separable Banach space $X$ 
embeds isometrically into $C([0,1])$ 
(see Theorem 1.4.3 in \cite{AK}).   Let 
$$
j:X\rightarrow C([0,1])
$$
denote such an embedding and consider the image $j(K)\subset C([0,1])$. It is clear by definition that
\begin{equation}
 \label{to1}   d_n^r(j(K))_{C([0,1])} \geq d_n(j(K))_{j(X)} = d_n(K)_{X},
\end{equation}
since $\mathcal{X}(j(K)) \subset j(X)$. Using the Hahn-Banach theorem, it follows  that
\begin{equation}
\label{PL}
    d^n(j(K))_{C([0,1])} =d^n(j(K))_{j(X)}= d^n(K)_X,
\end{equation}
since any linear functionals on $j(X)$ can be extended to $C([0,1])$.
It follows from \eref{PL} and Proposition \ref{kolmogorov-gelfand-proposition}, which ,  in this particular case states that $d_n(j(K))_{C([0,1])} \leq d^n(j(K))_{C([0,1])}$,  that 
\begin{equation}
  \label{to2}   d_n(j(K))_{C([0,1])} \leq d^n(K)_X.
\end{equation}
Relations \eref{to1}, \eref{to2}, and \eref{to3} give that
\begin{equation}
    d_n^r(j(K))_{C([0,1])} \geq d_n(K)_X \geq  C\sqrt{n}d^n(K)_X \geq C\sqrt{n}d_n(j(K))_{C([0,1])},
\end{equation}
that is, \eref{sharp} holds for the set $j(K)$ and the space $C([0,1])$.
\hfill $\Box$

\subsection{Comparison between the constrained Kolmogorov widths $d_n^{c}(K)_X$ and the Kolmogorov widths}
\label{SS:directcomparisons} 
Since the constrained widths $d_n^{c}(K)_X\geq d_n^{r}(K)_X$, 
the results of the previous  section, see Theorem \ref{T:sharp},  show that  in some cases the constrained widths of a compact set $K$ can be considerably larger than the  Kolmogorov widths of $K$.  There is a natural question of whether there is any direct comparison between these two widths that holds for all Banach spaces $X$  and sets $K\in\cK( X)$.  For example, do there exist  sequences $(m_n)_{n\ge 0}$, and $(\lambda_n)_{n\ge 0}$, such that
\be 
\label{direct}
d_n^c(K)_X\le \lambda_n d_{m_n}(K)_X,\quad n\ge 1,
\ee
holds for all $X$ and   $K\in \cK(X)$. 
 The next theorem shows that there is no such direct comparison  of the form \eref{direct} that holds for Hilbert spaces $X$ and  convex, centrally symmetric compact subsets $K\subset X$.

\begin{theorem}
\label{T:nocomparison}
For every $\varepsilon>0$ and every $m \geq n\geq 2$, there is a convex centrally symmetric compact set $K=K(\e,n,m)\subset \ell_2$, such that 
$$
d_n(K)_{\ell_2}<\varepsilon, \quad \hbox{and}\quad d^c_m(K)_{\ell_2}>1.
$$

\end{theorem}
\vskip .1in
\noindent
{\bf Proof:} 
We fix $m\geq n\geq 2$ and $\varepsilon>0$, and  choose $m+1$   vectors  
$v_1,\ldots, v_{m+1}\in \R^n$ such that $\span \{v_1, \ldots,v_{m+1}\}=\R^n$ and  
\be
\label{distconv}
\dist(v_i,\text{conv}(\pm v_j, j\neq i))
> 2,\quad i=1, \ldots,m.
\ee
One can construct such a set of vectors by taking quasi uniformly
spaced points on a sphere in $\R^n$ with sufficiently large radius.  Let $e_j$, $j\ge 1$, be the canonical basis for $\ell_2$ where $e_j$ has $j$-th coordinate one and all other coordinates zero.
For each $i=1,\dots,m+1$, we define  the vectors $w_i\in\ell_2$ with the property 
\be
\label{defw}
w_i\cdot e_j = \begin{cases}
    v_i\cdot e_j, & j \leq n,\\
    \epsilon, & j = n + i,\\
    0, & \text{otherwise}.
\end{cases} 
\ee
In other words, $w_i$ has its first $n$ coordinates given by $v_i$, its entry in coordinate $n+i$ is $\e$,   and  all of the remaining coordinates are zero. 
We then  consider the  centrally symmetric and convex set
\be 
    K:=K(\e,n,m) := \text{conv}\left\{\pm w_1, 
    \pm w_2,\ldots,\pm  w_{m+1}
    \right\} \subset \ell_2.
\ee 
 We write any element $z\in\ell_2$ as $z=(x;y)$, where
$x\in\R^n$ is the  vector consisting of its  first $n$ coordinates and $y$ is the vector consisting of its remaining coordinates.  In particular, any element $z\in K$ is of this form with
\be\label{first}
x=\sum_{i=1}^{m+1} c_iv_i, \quad y=(y_i)_{i\ge 1},\quad \hbox{where}\quad y_i=\varepsilon c_i,\quad 1\leq i\leq m+1,
\quad y_i=0\quad \hbox{for}\quad i>m+1,
\ee
and the coefficients $c_i$ satisfy the inequality
\be\label{nn}
\sum_{i=1}^{m+1}|c_i|^2\le \sum_{i=1}^{m+1}|c_i|\leq 1.
\ee
We define the $n$ dimensional subspace $V^*:=\span\{e_1,\dots ,e_n\}$. Then, from \eref{nn}, we have
\be
\dist(z,V^*)_{\ell_2}\le\e,\quad z\in K,
\ee
and therefore
\be 
\label{dn}
d_n(K)_{\ell_2} \le \e.
\ee

We now want to show that the constrained width $d_m^c(K)_{\ell_2}$  is large.   
  Let $V$ be any  $m$-dimensional subspace of $\ell_2$, where $m\geq n$.  Observe  that  the orthogonal projection   of $V$ onto the space 
$\span\{(0;e_1), \ldots,(0;e_{m+1})\}$, 
 where $0\in\R^n$, is an $m$ dimensional space.
 If  $z=(x;y)\in V$, then there are coefficients $a_i$, $i=1, \ldots,m+1$, depending only on $V$, not all zeroes, such that
$$
a_1y_1+\ldots+a_{m+1}y_{m+1}=0.
$$
Now, if $z$ is also in $K$, then $z=\sum_{i=1}^{m+1}c_iw_i$,   where the $c_i$'s satisfy \eref{nn}.   Combining this with \eref{first} gives that 
\be\label{second}
\sum_{i=1}^{m+1}a_ic_i=0.
\ee
 Let $j\in\{1,\dots,m+1\}$ be chosen so  that $|a_j|$ is the biggest among all these $|a_i|'s$. Then it follows from \eref{second} that
$$
|a_jc_j|=\left|-\sum_{i=1, i\neq j}^{m+1}a_ic_i\right|\leq \sum_{i=1, i\neq j}^{m+1} |a_i||c_i|\leq |a_j|
\sum_{i=1, i\neq j}^{m+1} |c_i|,
$$
and therefore
\be 
\label{boundcj}
|c_j|\leq \sum_{i=1, i\neq j}^{m+1} |c_i|\leq 1-|c_j|,
\ee
where we have used \eref{nn}.
Thus, 
\be
\label{musthave}
|c_j| \leq 1/2.
\ee
 In other words,  there is $j\in\{1,\dots,m+1\}$ depending only on $V$, such that if $z=\sum_{i=1}^{m+1}c_iw_i\in K\cap V$, then  \eref{musthave} must hold.

Consider now the approximation of the element $f:=w_j\in K$, by any element $z=(x;y)\in  K\cap V$. We write $z=\sum_{i=1}^{m+1}c_iw_i$ as in
\eref{first}.
We calculate using \eref{distconv}
\begin{eqnarray}
\|f-z\|_{\ell_2} &\ge& \|(1-c_j)v_j-\sum_{i=i\neq j}c_iv_i\|_{\ell_2(\R^n)} 
    =(1-c_j)\left\|v_j -\sum_{i=1, i\neq j}^{m+1}\frac{c_i}{1-c_j}v_i\right\|_{\ell_2(\R^n)}\nonumber\\
    &\geq &(1-c_j)\dist(v_j,\text{conv}(\pm v_i, i\neq j)) >1
    \end{eqnarray}
since $1-c_j \geq 1/2$ and $\sum_{i\neq j} |c_i| \leq 1-|c_j| \leq 1-c_j$. This proves the theorem.\hfill $\Box$

\vskip .1in

The example given in Theorem \ref{T:nocomparison} is rather exotic. One may ask what happens for well known compact sets $K$. 
The following lemma shows a  result of this type when $K$ is the unit ball in $\ell_1(\R^m)$.

\begin{lemma}
\label{L2}
Let   $m\ge 2$ and
$$
    U(\ell_1(\R^m)):= \{x\in \mathbb{R}^m,~\|x\|_{\ell_1(\R^m)}\leq 1\}\subset X=\ell_\infty(\R^m),
$$
be  the  unit ball in $\ell_1(\R^m)$. Then,  for $n<m$,  we have  that
\begin{equation}\label{3.2}
d_n(U(\ell_1(\R^m)))_{\ell_\infty(\R^m)} \leq C\sqrt{\frac{1 + \log(m/n)}{n}}, \quad \hbox{while}\quad d_n^c(U(\ell_1(\R^m)))_{\ell_\infty(\R^m)} \geq 1/2.
\end{equation}
\end{lemma}
{\bf Proof:} The estimate for the Kolmogorov width is a classical result, see \cite{K1,K2}. To prove the result for $d_n^c(K)_{\ell_\infty(\R^m)}$, let   $V\subset \R^m$ be any linear space of dimension $n<m$,  and suppose that $v_i\in V\cap U(\ell_1(\R^m)) $   approximates $e_i\in U(\ell_1(\R^m))$, $i=1, \ldots,m,$ to accuracy $<1/2$. That is, 
$$
\|e_i - v_i\|_{\ell_\infty(\R^m)}< 1/2, \quad  \hbox{and}\quad \|v_i\|_{\ell_1(\R^m)} \leq 1, \quad i=1, \ldots,m.
$$
Then the components $v_{ij}$ of the vectors $v_i$ satisfy $v_{ii} > 1/2$ and 
$$
\sum_{j=1,\,j\neq i}^m |v_{ij}| < 1/2, \quad i=1, \ldots,m.
$$
But this implies that the matrix with  rows $v_i$ is diagonally dominant and thus non-singular.   So, the vectors $v_i$,   $i=1,\dots,m$, are linearly independent,  and therefore cannot lie in a subspace of dimension $n<m$.
\hfill $\Box$

 Despite the negative results given above, we can still derive some weak  comparisons.
By combining Lemma \ref{L1} with the Carl's inequality,  we  obtain the following  weak indirect comparison.

\begin{lemma}
\label{WT}
    Let $X$ be a Banach space and $K\in \cK(X)$. Then,  for each  $\alpha>0$ there is a constant $C=C(\alpha)$such that
 $$
 \max_{k=1, \ldots,n}\left\{k^\alpha d^c_{2^k}(K)_X\right \}\leq C(\alpha)\max_{m=1, \ldots,n}\left \{m^\alpha d_{m-1}(K)_X\right \}.
 $$   
\end{lemma}
\noindent
{\bf Proof:} From Carl's inequality (see e.g. \cite{carl1981entropy}), we know that for any $\alpha>0$, we have
$$
 \max_{k=1, \ldots,n}\left\{k^\alpha\e_k(K)_X\right \}\leq C(\alpha)\max_{m=1, \ldots,n}\left \{m^\alpha d_{m-1}(K)_X\right \}.
 $$
 Lemma \ref{L1}, part (v) allows us to  replace $\e_k(K)_X$ by $2^{-1}d_{2^k}(K)_X$ and obtain the conclusion of the lemma.

\hfill $\Box$

\bigskip


The results presented  so far  show that, in general, the constrained Kolmogorov widths $d_n^c(K)_X$ can decay much slower than the classical  Kolmogorov widths $d_n(K)_X$.
This suggests that one needs to put additional requirements on the set $K$ in order to achieve a meaningful comparisons between its  constrained, or $\gamma$-constrained widths and its Kolmogorov widths. 
The remaining sections of this paper present classical settings of Banach spaces $X$ and compact sets $K$, where the asymptotic decay of the $\gamma$-constrained widths
$d_n^{c,\gamma}(K)_X$, for a fixed value of $\gamma$,   is the same as that of the Kolmogorov widths  $d_n(K)_X$.  In other words, in these settings, there is no essential 
loss in the efficiency of approximation when invoking the constraint.

\section{Widths of interpolation spaces}
\label{SSS:interpolation}

 A general family of model classes $K$ are those generated by  interpolation spaces. We study the constrained Kolmogorov widths of these classes in this section. Before going into details, 
 we  recall the definition of  the interpolation spaces $X_{\theta,q}$, $0<\theta<1$, $0<q\leq \infty$,
in the special case of the real method of interpolation. A comprehensive treatment of interpolation spaces can be found in the books \cite{BL,BS}. 


Throughout this section, we let $X$ be a quasi-Banach space  with a quasi-norm $\|\cdot\|_X$, and  $Y$ 
be a linear subspace of $X$ which is itself a space with a quasi-norm $\|\cdot\|_Y$.  The terminology quasi-Banach space for $X$ means that in place of the triangle inequality, we have
\be
\label{quasi}
\|f+g\|_X\le C^*(\|f\|_X+\|g\|_X),\quad f,g\in X,
\ee
with $C^*$ a constant depending only on $X$.
We need this concept below when dealing with Besov spaces.

We  assume that the embedding inequality
\be 
\label{embedY}
\|f\|_X\le C_Y\|f\|_Y,\quad f\in Y,
\ee
is valid and that the unit ball $U(Y)$ is a compact subset of $X$.

The K-functional for the pair $(X,Y)$ is defined for each $f\in X$ by
\be 
\label{Kfunctional}
K(f,t;X,Y):= \inf_{g\in Y}\left\{ \|f-g\|_X+t\|g\|_Y\right\},\quad t>0.
\ee
This is a concave increasing function of $t>0$.  By taking $g=0$, we see
that 
\be 
\label{Kbd}
K(f,t)\le \|f\|_X, \quad 0<t<\infty.
\ee
The K-functional measures how well $f$ can be approximated by elements $g$ of $Y $ with a control on $\|g\|_Y$.
From the compactness of $U(Y)$ in $X$, it follows that for each $t>0$, there is a $g_t\in Y$ such that
\be
\label{gt}
\|f-g_t\|_X+t\|g_t\|_Y =K(f,t;X,Y),\quad 0\le t<\infty.
\ee

The K-functional can be used to define new Banach spaces $Z$ that are intermediate to $X$ and $Y$ as follows.  If $0<\theta<1$ and $0<q< \infty$, we define  the interpolation space
$X_{\theta,q}:=(X,Y)_{\theta,q}$ as the set of all $f\in X$ for which
\be 
\label{intspace}
 \|f\|_{(X,Y)_{\theta,q}}:=\left\{\int_0^\infty [t^{-\theta}K(f,t;X,Y)]^q\frac{dt}{t}\right\}^{1/q}
\ee
is finite.
 When $q=\infty$, we use the same definition with the $L_q(\frac{dt}{t})$ quasi-norm replaced by the $L_\infty$ norm.
The expression $\|f\|_{(X,Y)_{\theta,q}}$ is a norm when $\|\cdot\|_X$  and $\|\cdot\|_Y$ are norms and   $q\ge 1$ and it is a quasi-norm in all cases.

 In what follows, we shall use the following known properties of interpolations spaces.
\begin{remark}
\label{R1}
Whenever $X$ and $Y$ are as above, we have the embeddings

 \be
\label{R11}
\|f\|_{X_{\theta,q}}\leq C\|f\|_{X_{\theta',q'}},  
\ee
provided  $0<\theta<\theta'<1$  and $0<q,q'\le\infty$,  or $\theta=\theta'$ and  $0<q'<q\le \infty$. Here the constant $C$ depends only on the parameters $\theta,\theta',q$ and $q'$.
\end{remark}
\vskip .1in
\noindent{\bf Proof:}  See (7.7) of Chapter 6 in \cite{DL}.
 \hfill $\Box$

The reiteration theorem (see Theorem 7.3 in Chapter 6 of \cite{DL},  or Theorem  2.4, Chapter 5 in \cite{BS})  shows that the   spaces $X_{\theta,q}:=(X,Y)_{\theta,q}$, $0<\theta<1$, $0< q\le\infty$, are an interpolation family. 
We shall use this fact in the following special case.
\begin{lemma}
    \label{LKF}
If $0<\theta_0,\theta<1$ and $0<q_0,q\le \infty$, then
\be
\label{reiteration}
(X,X_{\theta_0,q_0})_{\theta,q}=X_{\theta',q},\quad \theta':= \theta \theta_0,
\ee
with equivalent quasi-norms,  where the constants in the equivalency depend on $\theta,\theta_0,q,q_0$.  
\end{lemma}
  This lemma was proved in Theorem 2.4, Chapter 5 in \cite{BS} in the case $1\leq q,q_0\leq \infty$.  The case when either $q$ or $q_0$ is less than one is proved in \cite{DPS}.

\subsection{Upper bounds for Kolmogorov widths of interpolation spaces}
\label{SS:ubKw}

We now explain  how the K-functional is useful in proving results in approximation theory.  Let $X$ and $Y$ be as above  and  suppose that we know the Kolmogorov width   $d_n(U(Y))_X$,  where 
\be 
\label{defUY}
U(Y):=\{g\in Y:\ \|g\|_Y\le 1\}.
\ee
We claim that from this knowledge we can  give an upper bound for the Kolmogorov width of the unit ball of any of the interpolation spaces $X_{\theta,q}:=(X,Y)_{\theta,q}$.

\begin{lemma} 
\label{LL1}
There is an absolute constant $C>0$, such that for every $0<\theta<1$ and  $0 < q\le \infty$, we have
\be
\label{kolint1}
d_n(U(X_{\theta,q}))_X
\le C[d_n(U(Y))_X]^\theta,\quad n\ge0.
\ee
In addition, for any $0<\theta<\theta_0<1$ and   $0 < q,q_0\le \infty$, we have
 \be 
 \label{rkolint}
d_n(U(X_{\theta,q}))_X\leq 
C(\theta_0,\theta,q,q_0) [ d_n(U(X_{\theta_0,q_0}))_X]^{\frac{\theta}{\theta_0}}, \quad n\ge0.
\ee    
\end{lemma}
\noindent
{\bf Proof:}
Let us fix $n\geq 0$ and denote by 
$$
\delta:=d_n(U(Y))_X.
$$
Then $\delta\le C_Y$ because of \eref{embedY}.  
Let   $V$ be an $n$-dimensional Kolmogorov space for $U(Y)$. If such a Kolmogorov space does not exist, we use the $\e$ argument of Remark \ref{R:firstremark}. Fix any $f\in X$ and first choose $g:=g_{\delta}\in Y$ satisfying \eref{gt} with $t=\delta$ and then a best approximation $v\in V$ to $g$.  We then have
\be 
\label{Kolint}
\|f-v\|_X\le C^*(\|f-g\|_X+\|g-v\|_X)\le C^*(\|f-g\|_X+ \delta\|g\|_Y) = C^*K(f,\delta;X,Y).
\ee
Now suppose $\theta \in (0,1)$.  If $f\in U(X_{\theta,\infty})$,  we have $K(f,\delta;X,Y)\le \delta^\theta$ and $\|g\|_Y \leq \delta^{\theta - 1}$.  It follows that for any   $0 < q\le \infty$ the unit ball of $X_{\theta,q}$ satisfies
\be
\nonumber
d_n(U(X_{\theta,q}))_X
\le 
C d_n(U(X_{\theta,\infty}))_X
\le C \delta^\theta=C[d_n(U(Y))_X]^\theta,\quad n\ge0,
\ee
 where we have used Remark \ref{R1} to see that $U(X_{\theta,q})$ is contained in a ball of radius $C$ in $X_{\theta,\infty}$.
This completes the proof of 
\eref{kolint1}. 

Next, we prove \eref{rkolint}. We fix $\theta_0$ and $q_0$ and use the notation $Y_0:=X_{\theta_0,q_0}$.  
 It is known that the unit ball $U(Y_0)$  is a compact subset of $X$, see Theorem 3.8.1 in \cite{BL}.
Therefore, we can use
   use \eref{kolint1} with $Y_0$ in place of $Y$,  and obtain
  \be 
 \label{rKolint}
d_n(U((X,Y_0)_{\alpha,q}))_X\le  
C[d_n(U(Y_0))_X]^\alpha.
\ee
On the other hand, we know from Lemma \ref{LKF} that
\be
\label{knowre}
(X,Y_0)_{\alpha,q}= (X,Y)_{\alpha \theta_0,q}
\ee
with equivalent norms. If we take $\alpha=\theta/\theta_0$
 and use  \eref{knowre} in \eref{rKolint} we have proved \eref{rkolint}.
\hfill $\Box$

\subsection{Upper bounds for $\gamma$-constrained Kolmogorov widths of interpolation spaces}
\label{SS:lbKb}
In this section, we extend the results of the previous subsection to  the $\gamma$-constrained Kolmogorov widths. 

\begin{theorem}
\label{T:Kconst}
 Let $X,Y$ be as above and let $\gamma\ge 1$.   Then, for any $0<\theta<1$,  and any $  0 < q\le \infty$, there are constants $C_0,C_1$ depending only on $\theta$ and $q$,  such that 
 \be 
 \label{TKC}
 d_n^{c,\tilde\gamma}(U(X_{\theta,q}))_X \le C_0 (d_n^{c,\gamma}(U(Y))_X)^\theta,\quad n\ge 0,
 \ee
 provided that $\tilde \gamma  >C_1\gamma$.

 \end{theorem}
 \vskip .1in
 \noindent
 {\bf Proof:}   Throughout this proof, $C$ denotes a constant that depends at most on $\theta$,  and $q$ and its value may change from line to line. 
 We fix $n$ and use the abbreviated notation
  $$
  \delta_n:=d_n^{c,\gamma}(U(Y))_X.
  $$
 It follows from the definition of $\delta_n$ that given any $\e>0$,  there is a linear space $V$   of dimension $n$  with the property  that for any $h\in Y$, there is a $v\in V \cap Y$ such that
\be
\label{Proph}
\|h-v\|_X\le (\delta_n+\e)\|h\|_Y,\quad {\rm and}\quad \|v\|_Y\le \gamma \|h\|_Y.
\ee

  To prove \eref{TKC},  we fix $\theta $ and $q$,   and begin as in the derivation of \eref{Kolint}. Recall that $X_{\theta,q}$ is embedded in
  $X_{\theta,\infty}$ and (see Remark \ref{R1})
  \be 
  \label{embed111}
  \|f\|_{X_{\theta,\infty}}\le C\|f\|_{X_{\theta,q}}.
  \ee
  Given any $f\in X_{\theta,q}$, we take $g:=g_{\delta_n}$ for $f$
and an appropriate $v\in V\ $ satisfying \eref{Proph} for $g$ in place of $h$.  This gives
\begin{eqnarray} 
\label{Kolint11}
\|f-v\|_X&\!\!\!\le \!\!\!& C^*(\|f-g\|_X+\|g-v\|_X)\le C^*(\|f-g\|_X+ (\delta_n+\e)\|g\|_Y)\ \nonumber\\
&=&\!\!\!  C^*K(f,\delta_n;X,Y) +{C^*}\e\|g\|_Y\le C\delta_n^\theta +{C^*}\e\|g\|_Y\le  C\delta_n^\theta +C\e\delta_n^{\theta-1} \le C\delta_n^\theta,
 \end{eqnarray}
provided that $\e$ is chosen small enough (depending only on $n$).
Note  that   the function $v$ also satisfies 
\be
\label{boundYnorm}
\|v\|_Y \le \gamma \|g\|_Y \le  C\gamma \delta_n^{\theta-1}.
\ee

We want next to bound $\|v\|_{X_{\theta,q}}$, and thus we concentrate on estimating $K(t):=K(v,t;X,Y)$. 
On one hand, we have 
\be
\label{firstKest}
K(t)\le t\|v\|_Y\le C\gamma t \delta_n^{\theta-1}, \quad \hbox{when}\quad 0<t\le \delta_n.
\ee 
On the other hand, we also have the bound
\begin{eqnarray}
\label{secondKest}
K(t)&\le& \|v-g_t\|_X+ t\|g_t\|_Y\le C^*(\|f-g_t\|_X+\|f-v\|_X)+ t\|g_t\|_Y\nonumber\\
&\le & C^*K(f,t;X,Y) +{C^*}\|f-v\|_X\le C^* K(f,t;X,Y)+ C\delta_n^\theta,\quad \hbox{when}\quad t>\delta_n,
\end{eqnarray}
where we have used \eref{Kolint11}.

We can use these two estimates to bound $\|v\|_{(X,Y)_{\theta,q}}$ as follows.  
\vskip .1in
\noindent
{\bf Case 1:} $q=\infty$. In this case,   we have $K(f,t;X,Y)\leq t^\theta$, $t>0$.   Therefore, from \eref{firstKest} we have
\be
\label{infinity1}
K(t)\le \gamma t \delta_n^{\theta-1}\le C\gamma t^\theta, \quad 0<t\le \delta_n.
\ee 
Also, from \eref{secondKest} we have 
\be 
\label{infinit2}
K(t)\le Ct^\theta +C\delta_n^\theta\le Ct^\theta,\quad t>\delta_n.
\ee
Since $\gamma\ge 1$, this shows that $\|v\|_{\theta,\infty}\le C\gamma$,  and  completes the proof in this case.

 \vskip .1in
\noindent
{\bf Case 2:} $0< q<\infty$. We write
\be 
\label{sum}
\|v\|^q_{X_{\theta,q}}=\int_0^{\delta_n} [t^{-\theta }K(t)]^q\frac{dt}{t}+\int_{\delta_n}^\infty  [t^{-\theta} K(t)]^q\frac{dt}{t}=:I_1+I_2.
\ee
From \eref{firstKest}, we obtain
\be
\label{f1}
I_1= \int_0^{\delta_n} [t^{-\theta}K(t)]^q\frac{dt}{t}\le C^q \gamma^q \delta_n^{q(\theta-1)}\int_0^{\delta_n} t^{(1-\theta)q}\frac{dt}{t}= C^q\frac{\gamma^q}{(1-\theta)q}\le C^q\gamma^q.
\ee
For $I_2$, from \eref{secondKest},   we have the bound
\begin{eqnarray}
\label{f2}
I_2&=&\int_{\delta_n}^\infty [t^{-\theta}K(t)]^q\frac{dt}{t}\le \int_{\delta_n}^\infty \left[t^{-\theta}K(f,t;X,Y) + Ct^{-\theta}\delta_n^\theta  \right]^q\frac{dt}{t}\nonumber
\\ 
&\leq&2^q\int_{\delta_n}^\infty [t^{-\theta}K(f,t;X,Y)]^q \frac{dt}{t} + (2C)^q\delta_n^{\theta q}\int_{\delta_n}^\infty t^{-\theta q}\frac{dt}{t} \nonumber
\\
&\leq& 2^q\|f\|^q_{X_{\theta,q}}
+\frac{(2C)^q}{ \theta q} 
\leq C^q\gamma^q,
\end{eqnarray}
where in the last inequality we recall that $\gamma\ge 1$.

Combining the bounds we have for $I_1$ and $I_2$, we arrive at
\be
\|v\|_{X_{\theta,q}}\le C\gamma,
\ee 
 where now $C$ is a fixed constant depending only on $\theta$ and $q$. This completes the proof of
 the theorem. \hfill $\Box$

\begin{theorem} 
\label{T:Konstnew}
Let $X,Y$ be as above,  and  let $\gamma\ge 1$. Then, for any  values $0<\theta<\theta_0<1$ and $0< q_0, q\le \infty$, we have that 
 \be 
 \label{rcKolint}
d_n^{c,\bar\gamma}(U(X_{\theta,q}))_X
\le  {C_0}[ d_n^{c, \gamma}(U(X_{\theta_0,q_0}))_X]^{\frac{\theta}{\theta_0}}, \quad n\ge 0,
\ee
provided that  $\bar \gamma>{c_1}\gamma$, where $C_0$, $C_1$, depend  only on $\theta,\,\theta_0,\,q,\,q_0$.   
\end{theorem}
 \noindent
 {\bf Proof:} The proof of this theorem is similar to  the proofs of Theorem \ref{T:Kconst} and Lemma \ref{LL1},  and is left to the reader.\hfill $\Box$

\section{Widths of approximation spaces}
\label{SS:approxclasses}

Approximation classes are a major subject in approximation theory whose aim is to characterize all elements of a quasi-Banach space $X$ that are approximated with
a prescribed rate  by a specified method of approximation.  To describe these approximation classes in the setting of interest to us, we fix the quasi-Banach space $X$, where we measure approximation error,  and we fix any sequence $\cX:=(X_n)_{n\ge 0}$ of   nested linear subspaces $X_n\subset X_{n+1}\subset X$,  with $\dim(X_n)=n$   and $\bigcup_{n=0}^\infty X_n$ dense in $X$.  Here, we define
$X_0:=\{0\}$.  
Given any $f\in X$, we have the approximation error
\be 
\label{aerror}
E_n(f)_X:=E(f,X_n)_X:= \inf_{g\in X_n}\|f-g\|_X,\quad n\ge 0.
\ee
Note that $E_0(f)_X=\|f\|_X$.

 Consider any non-increasing  sequence ${\bf \sigma}:=(\sigma_n)$ of positive numbers $1=:\sigma_0 \geq\sigma_1\geq\ldots$,
\begin{equation}
\label{seq}
 \sigma_n\searrow 0, \quad \hbox{as}\quad n\to \infty.
\end{equation}
Typical examples of such sequence are
$\sigma_n=(n+1)^{-\alpha}$, $n\ge 0$,  with $\alpha>0$, or $\sigma_n=a^n$, $n\ge0$, where $0<a<1$. 
We define the approximation class
\begin{equation}
\label{aspace}
\cA(\sigma,\cX):= \cA(\sigma,\cX)_X:=\{f\in X:\,\,E_n(f)_X\leq M \sigma_n, \,n\ge 0,\,\,\hbox{for some}\,\, M=M(f)\ge 0\}.
\end{equation} 
Then, $\cA(\sigma,\cX)$ is clearly a linear subspace of $X$, which we equip  with the
norm 
 \be
 \label{defAsigma}
 \|f\|_{\cA(\sigma,\cX)}:=\sup_{n\geq 0}\sigma_n^{-1} E_n(f)_X.
\ee
 The unit ball
$$
U(\sigma,\cX):=\{f\in X:\,\,\|f\|_{\cA(\sigma,\cX)}\leq 1\}
$$
is a compact subset of $X$.

We have the following

\begin{theorem}
\label{Pin}
For any sequence $\sigma=(\sigma_n)$   satisfying \eref{seq} and  nested linear spaces $\cX=(X_n)_{n\geq 0}$,  $X_n\subset X$, with $\dim(X_n)=n$,   and $\bigcup_{n=0}^\infty X_n$ dense in $X$, the Kolmogorov width of  $U(\sigma,\cX)$ is 
$$
d_n(U(\sigma,\cX))_X=\sigma_n, \quad n=1,2, \ldots,
$$
and $X_n$ is a Kolmogorov space, i.e., it attains the Kolmogorov width.
\end{theorem}

\vskip .1in
\noindent
{\bf Proof:}
This theorem was proved in  \cite{Pinkus}, see Proposition 1.7 in Chapter 1, in the case of a normed linear  space $X$.  It also holds for a quasi-Banach space $X$. Indeed, the only part of the proof which required $X$ to be a Banach space is the fact that the Kolmogorov widths $d_n$ of the unit ball of a $k$-dimensional space are $\geq 1$ for $n < k$. This is proved using the Borsuk-Ulam theorem for Banach spaces $X$ in \cite{Pinkus}, but an alternative proof which also works for quasi-Banach spaces can be found for example in \cite{gerhold2025entropyapproximationkolmogorovnumbers}, Part 2, Section 7. \hfill $\Box$

We turn now to the $\gamma$-constrained approximation of these compact sets.

\begin{theorem}
 \label{T:CAC}
 Let $\sigma=(\sigma_n)_{n\ge0}$ decrease monotonically to zero and let $\cX=(X_n)_{n\ge 0}$ be any sequence of nested linear subspaces of $X$ with $\dim(X_n)=n$  and $\bigcup_{n=0}^\infty X_n$ dense in $X$.  Then,  we have
\be
\label{Tac}
d_n^{r}(U(\sigma,\cX))_X=d_n^{c,{2C^*}}(U(\sigma,\cX))_X=d_n^{c,\gamma_n}(U(\sigma,\cX))_X=d_n(U(\sigma,\cX))_X=\sigma_n, \quad  n\ge 0, 
\ee
where
\be 
\label{Tac1}
 \gamma_n:={C^*(1+\frac{\sigma_n}{\sigma_{n-1}}) \le{2C^*}},
 \ee
 {and $C^*$ is the constant in \eref{quasi}.}
\end{theorem}
\noindent{\bf Proof:}
Let us fix $f\in U(\sigma,\cX)$ and let $\varphi_k=\varphi_k(f)\in X_k$ be such that 
$$
E_k(f)_X=\|f-\varphi_k\|_X, \quad k\ge0.
$$
From the nestedness of the spaces $X_k$, $k\ge0$, we have 
$$E_k(\varphi_n)_X=0, \quad k\geq n.
$$
On the other hand,
for $k<n$,
$$
E_k(\varphi_n)_X\leq \|\varphi_n-\varphi_k\|_X\leq {C^*(\|\varphi_n-f\|_X+\|f-\varphi_k\|_X)\leq 
C^*(\sigma_n+\sigma_k)}.
$$
Thus, we have  
$$
\|\varphi_n\|_{\cA(\sigma,\cX)}=
\sup_{k\geq 0}\sigma_k^{-1}E_k(\varphi_n)_X\leq {C^*}\sup_{0\leq k<n}\frac{\sigma_n+\sigma_k}{\sigma_k} \le {C^*(1+\frac{\sigma_n}{\sigma_{n-1}})=:\gamma_n\leq 2C^*}.
$$
This means that $\varphi_n\in\gamma_nU(\sigma,\cX)\cap X_n$ and that $\varphi_n$ approximates $f$ to accuracy $d_n(U(\sigma,\cX))_X$, according to Theorem \ref{Pin}. 
Therefore, we have 
$$
d_n^{c,\gamma_n}(U(\sigma,\cX))_X\leq d_n(U(\sigma,\cX))_X=\sigma_n
$$
  The latter inequality, Lemma \ref{L1}, part (iv), and Theorem \ref{Pin}, give that  for $n\geq 0$, we have
\be
\nonumber
 \sigma_n=d_n(U(\sigma,\cX))_X\leq d_n^{r}(U(\sigma,\cX))_X\le {d_n^{c,2C^*}(U(\sigma,\cX))_X}\le d_n^{c,\gamma_n}(U(\sigma,\cX))_X=d_n(U(\sigma,\cX))_X\le \sigma_n. 
\ee
Hence, we must have equality everywhere in the above inequalities.
\hfill $\Box$

\subsection {Refined  approximation spaces}
\label{S:other}

 There are generalizations of approximation spaces which capture the  decay of approximation error in a more subtle way than the spaces $\cA(\sigma,\cX)_X$.  We next describe the best known of these generalizations. We now let $X$ be any quasi-Banach space.  In particular, we allow $X$ to be an $L_p$ space with $0<p<1$.
 
 As before, we denote by $\cX:=(X_n)_{n\ge 0}$,  a sequence of nested linear subspaces $X_n$ with $\dim(X_n)=n$   and  $\bigcup_{n=0}^\infty X_n$ dense in $X$.   We define   the spaces 
 $$\cA_q^r:=\cA_q^r(\cX,X),\quad  \hbox{for}\quad r>0, \quad 0<q< \infty,
 $$
 as the set of all functions $f\in X$ for which
 \be 
 \label{approxspaces}
 \|f\|_{\cA_q^r}:=\left\{\sum_{n\ge 0} [(n+1)^rE_n(f)_X]^q (n+1)^{-1}\right\}^{1/q}<\infty.
\ee
When $q=\infty$, we define 
$$\cA_\infty^r(\cX,X):=\cA((n+1)^{-r},\cX)_X,
$$
where $\cA((n+1)^{-r},\cX)_X$ is given  by  \eref{aspace} with its defined norm \eref{defAsigma}.
\begin{lemma}
\label{L:Aemb}
We have the embedding inequality
\be
\label{inc}
 \|f\|_{\cA^r_\infty}\le 2^{(r+1/q)}\|f\|_{\cA_q^r}, \quad 0<q<\infty.
\ee
\end{lemma}
\vskip .1in
\noindent
{\bf Proof:} Fix $f\in\cA_q^r$ and let $E_k:=E_k(f)_X$. Since the sequence $(E_k)_{k\ge 0}$ is non-increasing, for each $n\ge 0$, we have
\be
\label{embest1}
[(n+1)^rE_n]^q\le (n+1)2^{rq}[(k+1)^rE_k]^q (k+1)^{-1}, \quad n/2\le k\le n.
\ee
There are at least $(n+1)/2$ integers  $k\in [n/2,n]$.  If we sum \eref{embest1} over those integers, we obtain
 $$
 [(n+1)^rE_{n}]^q \le 2^{rq+1} \|f\|^q_{\cA_q^r}, \quad n\ge 0.
$$
 Taking the $q$-th root of this inequality proves the lemma.\hfill $\Box$

An important property of approximation spaces is that they are an interpolation family (see \cite{DPA}, Corollary 3.4, or \cite{DL}, Chapter 7). Namely, for the real method of approximation, we have that for
$0<r_1<r_2$ and any $0<q,q_1,q_2\le\infty$, 
and $0<\theta<1$, 
\be 
\label{ASinter}
\cA_q^s(\cX,X)  =(\cA_{q_1}^{r_1}(\cX,X),\cA_{q_2}^{r_2}(\cX,X))_{\theta,q},\quad s=(1-\theta)r_1+\theta r_2, 
\ee
with equivalent {quasi-}norms
\be
\label{equivnorms}
\|f\|_{\cA^s_q(\cX,X)}\asymp \|f\|_{(\cA_{q_1}^{r_1}(\cX,X),\cA_{q_2}^{r_2}(\cX,X))_{\theta,q}},\quad s=(1-\theta)r_1+\theta r_2, \ \  0<q\le \infty,
\ee
and constants of equivalence depending only on $r_0$, $r_1$, $q$, $q_1$, $q_2$, and $\theta$.

 
One of the main chapters of approximation theory is to characterize  the approximation classes 
for specific sequences $\cX=(X_n)_{n\ge 0}$ of nested $n$ dimensional linear spaces  whose union is dense in $X$.  We outline the typical way this is accomplished.
Given a sequence $\cX$, suppose that for some value of $r>0$, we can find a quasi-Banach space $Y_r\subset X$  for which the following two inequalities hold (see \cite{DPA} or Chapter 7 of \cite{DL} for details):

\vskip .1in
\noindent
{\bf Jackson Inequality:}
\be 
\label{Jackson}
E_n(f)_X\le C\|f\|_{Y_r}(n+1)^{-r},\quad n\ge 0,
\ee

\vskip .1in
\noindent
{\bf Bernstein Inequality:} 
\be 
\label{Bernstein}
\|g\|_{Y_r}\le C\|g\|_X (n+1)^r,\quad g\in X_n,\ n\ge 0.
\ee
\vskip .1in

Then, for any $0<s<r$ and $0<q\le\infty$,   the approximation space $\cA_{q}^s:=\cA_{q}^s(\cX,X)$ is characterized as an interpolation
space, that is
\be 
\label{JBinter}
\cA^s_{q}= (X,Y_r)_{\theta,q}, \quad \hbox{where}\quad  \theta:=s/r.
\ee
The approximation space norm is equivalent to the interpolation space norm with the constants of equivalency depending on $s$ and $q$.      
 
In particular, see \cite{DL}, Chapter 7, Theorem 9.3, or Corollary 3.3 in \cite{DPA}, one can take the space $Y_r$ in \eref{JBinter} to be $Y_r=\cA^r_\tau$ for any $s<r<\infty$ and any $0<\tau\leq \infty$, that is,
\be
\label{appr}
\cA^s_{q}= (X,\cA^r_\tau)_{\theta,q}, \ 0<q\le\infty,
\quad \theta:=s/r,
\ee
with equivalent norms, and the constants in the norm equivalence
depending only on the parameters.

The following theorem is of the same flavor as Theorem \ref{Pin}.
\begin{theorem}
\label{T:KwA}
Let $X$ be a {quasi-}Banach space and $\cX=(X_n)_{n\ge 0}$ be a sequence of nested subspaces of dimension $n$ in $X$ for which $\bigcup_{n=0}^\infty X_n$ is dense in $X$.   For any $s>0$, $0<q\le\infty$, we have that the Kolmogorov widths of the unit ball $U(\cA_q^s)_X:=U(\cA_q^s(\cX,X))$ satisfy
\be 
\label{Pinkus++}
   d_n(U(\cA_q^s))_X \asymp (n+1)^{-s},\quad n\ge 0,
\ee 
where the constants  in the asymptotic decay depend only on $s$ and $q$.  \end{theorem}
\vskip .1in
\noindent
{\bf Proof:}   We fix $n$.  It follows from Lemma \ref{L:Aemb} and  Theorem \ref{Pin}, that 
\be 
\label{TP}
d_n(U(\cA_q^s))_X\le C d_n(U(\cA_\infty^s))_X=C(n+1)^{-s}, \quad n\ge 0.
\ee

We now prove the lower bound  in \eref{Pinkus++}.  We already know  this lower bound   when $q=\infty$  from  Theorem \ref{Pin}.  
 Since the approximation spaces are an interpolation family,  it follows from  \eref{rkolint} and \eref{JBinter} that for $s':=3s/4$, we have
\be 
\label{lb1}
(n+1)^{-3s/4}=(n+1)^{-s'}=d_n(U(\cA^{s'}_\infty))_X
\le C_0[d_n(U(\cA_q^s))_X]^{s'/s}=C_0[d_n(U(\cA_q^s))_X]^{3/4},
\ee
where $C_0$ is the constant in \eref{rkolint}.
Thus, the lower bound in  \eref{Pinkus++}
holds with {the constant $C=C_0^{-4/3}$.}



\hfill $\Box$

\subsection{$\gamma$-constrained Kolmogorov widths of  refined approximation spaces}
\label{SS:cwas}

We turn now to analyzing the constrained widths of the unit ball $U(\cA_q^s)$ and prove the following theorem.  
\begin{theorem}
\label{T:cKwA}  
Let $X$ be a {quasi-}Banach space and $\cX=(X_n)_{n\ge 0}$ be a sequence of nested subspaces of dimension $n$ in $X$,   whose union is dense in $X$.

\noindent
{\rm (i)}  For any $s>0$ and $0<q\le\infty$, there is $\bar \gamma(s,q,X) \geq 1$ such that for all  $\gamma\geq \bar \gamma(s,q,X)$ the constrained Kolmogorov widths of the unit ball $U(\cA_q^s):=
U(\cA_q^s(\cX,X))$ satisfy
\be 
\label{Pinkus+}  
d_n^{c,\gamma}(U
(\cA_q^s))_X \asymp (n+1)^{-s},\quad n\ge 0,
\ee 
with constants of equivalency depending only on $s,q$. 

\noindent
{\rm (ii)} 
Additionally, if $s'<s$ and $ 0< q'\le \infty$, then there is a $\bar\gamma(s,s',q, q',X) \geq 1$ so that for all $\gamma \geq \bar\gamma(s,s',q, q',X)$ we have
\be 
\label{Pinkus1+} 
d_n^{c, \gamma}(U
(\cA_q^s))_{\cA_{ q'}^{s'}} \asymp  (n+1)^{-(s-s')},\quad n\ge 0,
\ee 
with the constants of equivalency depending on $s,s',q,q'$.
\end{theorem}

\vskip .1in
\noindent
{\bf Proof:}
It follows from  Theorem \ref{T:CAC} with $\sigma_n=(n+1)^{-s}$ that
\be
\label{csA}
(n+1)^{-s} =  d_n(U(\cA^s_\infty))_X =
{d_n^{c,2C^*}}(U(\cA^s_\infty))_X,\quad s>0 .
\ee
To prove the upper bound in \eref{Pinkus+}, we fix $s,q$ and  use the fact that the $\cA_q^s$ spaces are an interpolation family (see \eqref{appr}).  This together with
\eref{csA} and {Theorem \ref{T:Kconst}} gives
  for  $r=2s$ the upper bound
\be
\label{cwA}
   d_n^{c,\gamma}(U(\cA^s_q))_X
   \le C ({d_n^{c,2C^*}}(U(\cA^r_\infty)))^{s/r}_X\le C(n+1)^{-s},\quad n\ge 0,
\ee
provided $\gamma$ is sufficiently large (depending on $s,q$).

To prove the lower bound in \eref{Pinkus+}, we know that the $\gamma$-constrained width is always as large as the Kolmogorov width.  Hence, from Theorem \ref{T:KwA} we deduce that for all $\gamma$ we have
\be
C (n+1)^{-s}\le d_n(U(\cA^s_q))_X\le d_n^{c,\gamma}(U(\cA^s_q))_X,
\ee
as desired.
 
To prove  \eref{Pinkus1+} of (ii), we fix $q,q',s,s'$,  and we replace the space $X$ where we measure error by $X_0:=\cA_{ q'}^{s'}$.  
 It is known that the approximation spaces  satisfy  
 
\be
\label{newAspaces}
\tilde \cA_q^s:=\cA_q^{s-s'}(\cX,\cA_{ q'}^{s'})=\cA_q^{s}(\cX,X),\quad s>s' \ {\rm and} \ 0<q,q'\le \infty,
\ee
with equivalent norms depending only on the parameters. Indeed, \eref{newAspaces} follows from the fact that the  spaces $X_n$, $n\ge 0$, satisfy Jackson and Bernstein inequalities
 when approximation takes place in  $X_0$  (see Appendix \ref{A:JB}),  since for any $s_0>s$, we have
 $$
 \cA_\tau^{s-s'}(\cX,\cA_{ q'}^{s'})=(\cA^{s'}_{q'},\cA^{s_0}_q)_{\frac{s-s'}{s_0-s'},\tau} =\cA_\tau^{s}(\cX,X),
 $$
 where we have used Theorem 9.1 and equation (9.7) in Chapter 7
 from \cite{DL}.

From what we have already proved in \eref{Pinkus+} for this new approximation family
$\tilde\cA_q^s$ we know  there is a $\tilde\gamma$ such that 
\be
\label{newtilde}
d_n^{c,\tilde\gamma}(U(\tilde \cA_q^{s-s'}))_{X_0}\asymp (n+1)^{-(s-s')},\quad n\ge 0.
\ee
The conclusion of the theorem then follows from Lemma \ref{L1}, part (i) since we have 
$\mu U(\tilde \cA_q^{s})\subset U(\cA_q^s)\subset \lambda U(\tilde \cA_q^{s})$ for some $\lambda,\mu>0$.

\hfill $\Box$

\section{$\gamma$-constrained Kolmogorov widths of compact sets defined by smoothness}
\label{S:Besov}
As an illustration of how the results of this paper can be used to determine constrained widths,  we discuss, in this section,  the $\gamma$-constrained widths of compact sets described by a smoothness condition, when  the approximation error  in measured in   $L_p$,  $1\le p\le \infty$.

We start by recalling  the known results 
for constrained widths of smoothness classes in $L_p$.  These have been  obtained exclusively in the one dimensional case
 $d=1$.   Konovalov introduced    constrained widths in \cite{konovalov1984} and proved that
\be
 \label{Kon1}   d^c_n(U(W^s(L_\infty)))_{L_\infty} \geq Cn^{-\min\{s,2\}},\quad s\in \N,
\ee
where $U(W^s(L_\infty))$ is the unit ball of the one-dimensional periodic Sobolev space $W^s(L_\infty)$. This result implies  that there can be a substantial gap between $d_n^c(K)_X$ and $d_n(K)_X$ for certain compact sets $K\subset X$ for certain Banach spaces $X$, since it is well known that the Kolmogorov widths of $U(W^s(L_\infty))$  behave like $n^{-s}$ (see Theorem 1.1 in Chapter VII of \cite{Pinkus}, for instance). On the other hand, 
  there are Banach spaces $X$ and compact sets in $K\subset X$ for which the Kolmogorov and constrained Kolmogorov widths have the same behavior. For example, it was shown in \cite{BBDP}
that for
 $$
 K:=\{f\in W^1(L_p([0,1])):\,\|f'\|_{L_p([0,1])}\leq 1\}\subset L_2([0,1]), \quad 1\leq p\leq \infty,
$$
we have
$$
d_n(K)_{L_2([0,1])}\asymp d^{c}_n(K)_{L_2([0,1])}\asymp  n^{-1+(1/p-1/2)_+},  \quad n\ge 1.
$$

The picture is slightly different for  the $\gamma$-constrained widths. In spite of the result \eref{Kon1} for the constrained widths of $U(W^s(L_\infty))$, Babenko \cite{babenko1991} showed that for any $\gamma > 1$,  we have
\begin{equation}
    d_n^{c,\gamma}(U(W^s(L_\infty)))_{L_\infty} \leq C(\gamma, s)n^{-s}, \quad s\in \N.
\end{equation}
His work raises the interesting question of how the constant $C(\gamma,s)$ depends on $\gamma$ and $s$.
This question was addressed in \cite{malykhin2016relative}, where it was shown that for any finite $\gamma < \infty$, $C(\gamma,s)$  must grow with $s$.

In this section, we investigate the  $\gamma$-constrained Kolmogorov widths $d_n^{c,\gamma}(K)_X$ of unit balls $K$ of smoothness spaces in the more general setting of bounded Lipschitz domains $\Omega$ in $\R^d$, $d\geq 1$. 
The classical examples of smoothness spaces on $\Omega$
 are the   Lipschitz,  Sobolev, or  Besov spaces.  For the definition
 and properties of these spaces, we refer the reader to  \cite{DL} for the univariate case and \cite{BS,stein1970singular,triebel2006theory} for the multi-dimensional setting.

 We concentrate on Besov spaces, since they are both approximation spaces and interpolation spaces, so that we can apply the theory developed in the previous two sections on the $\gamma$-constrained Kolmogorov widthss. 
For the  most part, we are able to determine the asymptotic decay of these  widths. We show that they behave like the ordinary Kolmogorov widths,
provided the parameter $\gamma$ is chosen large enough. However, there are some exceptional cases where we do not yet know the asymptotic decay of $d_n^{c,\gamma}(K)_X$.  These exceptional cases occur when   the  Kolmogorov widths $d_n(K)_X$ of $K$ cannot be obtained using classical approximation methods (polynomials, splines, wavelet) and  are commonly referred to as the Kashin regime (cases where approximation based on probabilistic methods is needed).

\subsection{Besov spaces}
\label{SS:Besov}
We briefly recall the definition of Besov spaces 
and  some of  their properties.     The material in this section is taken for the most part from the papers \cite{DS,DP}. The reader may refer to
 \cite{DL} for  the univariate case when $\Omega$ is an interval.
 
 If $r$ is a positive integer, $0<p\le \infty $ and $f\in L_p(\Omega)$, we define the modulus of smoothness $\omega_r(f,\cdot)_{p}$ of $f$ by
\be 
\label{moduli}
\omega_r(f,t)_{p}:=\omega_r(f,t,\Omega)_{p}:= \sup_{|{\rm \bf h}|\le t} \|\Delta_{\rm \bf h}^r(f,\cdot)\|_{L_p(\Omega_{r{\rm \bf h}})}, \quad t>0,
\ee
where 
\be 
\label{rdiff}
\Delta_{\rm \bf h}^r(f,\cdot):=(-1)^r\sum_{k=0}^r (-1)^k\binom{r}{k} f(\cdot+k{\rm \bf h}),
\ee 
is the $r$-th difference of $f$ for ${\rm \bf h}\in\R^d$ and $\Omega_{\rm \bf h}:=\{{\rm \bf x}\in\Omega: [{\rm \bf x},{\rm \bf x}+{\rm \bf h}]\subset \Omega\}$.  Here
$[{\rm \bf x},{\rm \bf x}+{\rm \bf h}]$, with 
${\rm \bf x}$, ${\rm \bf h}\in\R^d$, denotes the line segment in $\R^d$ between ${\rm \bf x}$ and ${\rm \bf x}+{\rm \bf h}$,  and $|{\rm \bf h}|$ is the Euclidean norm of ${\rm\bf h}$.

If $s>0$ and $0< p,q\le  \infty$, then the Besov space
$B_q^s(L_p(\Omega))$ is defined as the set of all functions in $f\in L_p(\Omega)$ for which
\be 
\label{Bsemi}
|f|_{B_q^s(L_p(\Omega))}:= \left [\int_0^1 [t^{-s}\omega_r(f,t)_p]^q\frac{dt}{t}\right ]^{1/q}<\infty,\quad 0<q<\infty,
\ee 
where $r$ can be taken as any integer strictly bigger than $s$. When $q=\infty$, we replace the integral in \eref{Bsemi}  by a supremum. This is a (quasi-)semi-norm and we obtain the (quasi-)norm for $B_q^s(L_p(\Omega))$
by adding $\|f\|_{L_p(\Omega)}$ to it. 
While different choices of $r$ result in different (quasi-)semi-norms, the corresponding Besov (quasi-)norms are equivalent, provided $r>s$.
To fix matters, we define the Besov norm by taking  the value of $r=r(s)$ as the smallest integer strictly larger than $s$.  
The (quasi-) norm on this Besov space is then
\be 
\label{Bnorm}
\|f\|_{B_q^s(L_p(\Omega))}:= |f|_{B_q^s(L_p(\Omega))}+\|f\|_{L_p(\Omega)}.
\ee 

Let us make some remarks on some properties of these spaces. Note that 

\be
\label{Besovembed1}
\|f\|_{B_{q_1}^s(L_p(\Omega))}\le C\|f\|_{ B_{q_2}^s(L_p(\Omega))}, \quad 
0<q_2\le q_1\le \infty,
\ee
with the embedding constant $C$ depending only on the parameters.
In other words, the Besov  spaces get smaller as $q$ gets smaller.  

The effect of $q$ in the definition of the Besov spaces is subtle. We remark that when $p=q$ and $s$ is not an integer, the space $B^s_p(L_p(\Omega))$ is equivalent to the Sobolev space $W^s(L_p(\Omega))$, see \cite{adams2003sobolev}.  

%
%

\subsection{Equivalent description of Besov spaces as interpolation spaces or approximation spaces}
\label{SS:wBesov}

When analyzing  the $\gamma$-constrained Kolmogorov widths of finite balls of a Besov space, we use two important properties of these spaces. 
The first property is that the for any fixed $0< p\le\infty$, the Besov spaces $B_q^s(L_p(\Omega))$ are an interpolation family. Namely, for $0<s_1<s_2$ and $0<q,q_1,q_2\le\infty$, we have
\be
\label{Bi11}
\left (B_{q_1}^{s_1}(L_p(\Omega)),B_{q_2}^{s_2}(L_p(\Omega))\right )_{\theta,q} = B^s_q(L_p(\Omega)), \quad s=(1-\theta)s_1+\theta s_2,
\ee
with equivalent norms.
For a proof of this fact, see Corollary 6.2 in \cite{DP1} in the case of $\Omega$ being the unit cube in $\R^d$ and Corollary 6.8 in \cite{DS} for Lipschitz graph domains $\Omega$.
In the case $1\leq p\leq \infty$, for $0<\theta<1$ and $0<q\le \infty$, we have
\be
\label{Bi1}
B_q^{\theta r}(L_p(\Omega))= \left (L_p(\Omega),W^r(L_p(\Omega))\right )_{\theta,q},
\ee
 see \cite {DS}.
Moreover, the interpolation space norm is equivalent to the Besov space norm given in \eref{Bnorm}.

The second property we will use is that Besov spaces are approximation spaces.  Namely, if  $X=L_\tau(\Omega)$, $0<\tau\le \infty$, there is a family
$\cX=(X_n)_{n\ge 0}$ of nested linear spaces, with $\dim(X_n)=n$,  whose union is dense in $L_\tau(\Omega)$, such that
\be
\label{Ba1}
B_q^s(L_\tau(\Omega))= \cA_q^{s/d}(\cX,L_\tau(\Omega)), \quad s>0,\, \ 0<\tau\le \infty,~0<q\le \infty.
\ee
For example, it is shown in \cite{DP} that one may take $X_n$ as an $n$ dimensional linear space of piecewise polynomials on a partition of $\Omega$,  when $\Omega$ is the unit cube. For the more general case of bounded Lipschitz graph domains a similar result was given in \cite{DS}.  One may also take $X_n$ to be the linear space spanned by $n$ terms of a multiscale (wavelet) basis \cite{CDDmulti} under certain assumptions on $\Omega$.

\subsection{Upper bounds for $\gamma$-constrained Kolmogorov widths of Besov balls}
Let $\Omega$ be a bounded Lipschitz graph domain in $\R^d$.  We fix $L_p(\Omega)$, $1\le p\le \infty$, as the Banach space where approximation is to take place.  We use the notation  $U^s_{\tau,q}:=U(B^s_q(L_\tau(\Omega)))$  for  the unit ball of the Besov space $B_q^s(L_\tau(\Omega))$ for $s>0$ and $0<q,\tau\le\infty$.  Then $U^s_{\tau,q}$ is a compact set in $L_p(\Omega)$ if and only if
\be 
\label{Bcompact}
s>\left(\frac{d}{\tau}-\frac{d}{p}\right)_+.
\ee

\begin{theorem}
\label{T:Besovcwidth}
 For the set $U^s_{\tau,q}$, where $0< q,\tau\leq \infty$,  $s>d(1/\tau-1/p)_+$, $1\leq p\leq \infty$,  there is a $\gamma=\gamma(p,q,\tau,s)$ such that
\be 
\label{direct1}
d_n^{c,\gamma}(U^s_{\tau,q})_{L_p(\Omega)}\le C(p,q,s,\tau) (n+1)^{-(\frac{s}{d}-(\frac{1}{\tau}-\frac{1}{p})_+)}, \quad n\ge 0.
\ee
\end{theorem}
\vskip .1in
\noindent
{\bf Proof:}  Let us first consider the case when $\tau=p$, i.e., the smoothness is measured in the same space $L_p$ where we measure error.  We use the  fact that these Besov spaces are approximation spaces, see \eref{Ba1}.  The upper bound in this case therefore follows from \eref{Ba1}  and Theorem \ref{T:cKwA}. 

In the case $\tau>p$,  we have from what we just proved that
\be
\label{taup}
d_n^{c,\gamma}(U_{\tau,q}^s)_{L_p(\Omega)}\le 
C d_n^{c,\gamma}(U_{\tau,q}^s)_{L_\tau(\Omega)}\le C (n+1)^{-\frac{s}{d}},\quad n\ge 0,
\ee
provided $\gamma$ is sufficiently large. Indeed, let  $\delta_n:=d_n^{c,\gamma}(U_{\tau,q}^s)_{L_\tau(\Omega)}$.  We know that $\delta_n\asymp ( n+1)^{-s/d}$, $n\ge 0$.  If  $f\in U_{q,\tau}^s$, let $g\in X_n$ provide the bound
$$ 
\|f-g\|_{L_\tau}\le \delta_n \quad {\rm and} \quad \|g\|_{B_q^s(L_\tau(\Omega))}\le \gamma \|f\|_{B_q^s(L_\tau(\Omega))}.
$$
Then, $\|f-g\|_{L_p(\Omega)}\le C\|f-g\|_{L_\tau(\Omega)}\le C\delta_n$ because $\Omega$ is bounded. This proves \eref{direct1} in this case.

Finally, we consider the case $\tau<p$.  Let $s':=d(\frac{1}{\tau}-\frac{1}{p})$.  In this case, we know from \eqref{Ba1} that
\be
\label{charapprox}
B^s_q(L_\tau(\Omega)) = \cA^{s/d}_q(\cX,L_\tau(\Omega))~~\text{and}~~ B^{s'}_1(L_\tau(\Omega)) = \cA^{s'/d}_1(\cX,L_\tau(\Omega))
\ee
with equivalent norms. Therefore, if $f\in U_{q,\tau}^s$,   from  part (ii) of 
Theorem \ref{T:cKwA} for an appropriate $\gamma=\gamma(\tau,p,s,q)$, there is a function  $g\in X_n$  such that
\be 
\label{thereisg}
\|f-g\|_{B^{s'}_{1}(L_\tau(\Omega))}\le C(n+1)^{-(s-s')/d},\quad {\rm and}\quad \|g\|_{B^s_q(L_\tau(\Omega))}\le \gamma\|f\|_{B^s_q(L_\tau(\Omega))}.
\ee
It is standard fact that the Besov space $B^{s'}_{1}(L_\tau(\Omega))$ embeds into $L_p$, i.e., that $\|\cdot\|_{L_p}\le C\|\cdot\|_{B^{s'}_{1}(L_\tau(\Omega))}$.  This gives
\be 
\label{thereisg1}
\|f-g\|_{L_p}\le C(n+1)^{-(s-s')/d}\le C(n+1)^{-(\frac{s}{d}-(\frac{1}{\tau}-\frac{1}{p})_+)},
\ee
which proves \eref{direct1} in this last case. \hfill $\Box$

\subsection{Lower bounds for the $\gamma$-constrained Kolmogorov widths of Besov classes}
\label{SS:lbBesov}
Of course, the $\gamma$-constrained Kolmogorov widths   of Besov classes in $L_p(\Omega)$ are always larger than their Kolmogorov widths, i.e.,
\be 
\label{lbBesov}
d_n^{c,\gamma}(U^s_{\tau,q})_{L_p(\Omega)}\ge d_n(U^s_{\tau,q})_{L_p(\Omega)}.
\ee
The  Kolmogorov widths of many classical smoothness classes are known (see Chapters 13 and 14 of \cite{LGM}, or Chapter VII or \cite{Pinkus}). 
These unconstrained widths match the upper bound
in Theorem \ref{T:Besovcwidth} when $p \leq \tau$ or when $1\leq \tau,p\leq 2$ (see for instance Theorem 3.8 in Chapter 14 in \cite{LGM} for the one dimensional case and Lipschitz spaces). Outside of this regime, probabilistic methods are necessary to construct optimal Kolmogorov spaces \cite{Kashin} and we do not know the constrained widths in these cases. We leave this as an interesting open problem.

\vskip .2in

\section{Appendix} 
\label{S:App}

\subsection{Proof of Proposition \ref{kolmogorov-gelfand-proposition}}
\label{SS:A1}
   We want to to show that if  $\Omega$ is a compact topological space and $K$ is a compact, centrally symmetric subset of $C(\Omega)$, then we have
    \begin{equation}
    \nonumber
        d_n(K)_{C(\Omega)} \leq \delta_n(K)_{C(\Omega)} = d^n(K)_{C(\Omega)}, \quad  n\ge 1,
    \end{equation}
where $d^n$ and $\delta_n$ denote the Gelfand and linear widths, respectively. Since the left inequality is trivial and we also clearly have $d^n(K)_X \leq \delta_n(K)_X$ for any compact $K\subset X$ in any Banach space $X$ (see for instance \cite{Pinkus} or Chapter 13 in \cite{LGM}), we need only prove that 
\be 
\label{toproveGelfand}
\delta_n(K)_{C(\Omega)} \leq d^n(K)_{C(\Omega)},\quad n\ge 1.
\ee

   To prove \eref{toproveGelfand}, we fix $n\ge 1$ and let    $\delta > 0$ be arbitrary and fixed.   Let 
   {$\lambda_1,...,\lambda_n\in C(\Omega)^* 
   $} be chosen to realize the Gelfand widths up to error $\delta$, i.e.,  
    \begin{equation}\label{realizing-Gelfand-widths-equation}
        \sup\{\|f\|_{C(\Omega)}:~f\in K~\text{and}~\lambda_i(f) = 0, \ i=1,\dots,n\}\leq d^n(K)_{C(\Omega)} + \delta.
    \end{equation}
    Consider the continuous map $E:K\times \Omega\rightarrow \mathbb{R}^{n+1}$ given by
    \begin{equation}
        E(f,x) := (\lambda_1(f),...,\lambda_n(f),f(x)).
    \end{equation}
    For each $x\in \Omega$, the image
    \begin{equation}
        E(K,x) := \{E(f,x):~f\in K\}\subset \R^{n+1},
    \end{equation}
    is a compact, convex, centrally symmetric subset of $\mathbb{R}^{n+1}$. 
    
    Fix $\epsilon > 0$ and let $z := (d^n(K)_{C(\Omega)} + \delta + \epsilon)e_{n+1}$. Observe that by \eqref{realizing-Gelfand-widths-equation} it follows that $z\notin E(K,x)$. We define
    \begin{equation}
        z(x) := \arg\min_{y\in E(K,x)} \|y - z\|_{\ell_2(\R^{n+1})}
    \end{equation}
    to be the closest point in $E(K,x)$ to $z$. This point exists and is unique since $E(K,x)$ is compact and convex.
    It follows relatively easily from this that $z(x)$ is a continuous function of $x$. 
    
    Next, we observe that for every $x\in \Omega$ we have
    \be
    \label{eqq}
    z(x)\cdot e_{n+1} < z\cdot e_{n+1}.
    \ee
    Indeed,  since $0\in E(K,x)$ for every $x\in \Omega$,  by convexity we have that $(0-z(x))\cdot (z - z(x)) \leq 0$, and therefore
    \begin{equation}
    \label{eqq0}
        0 \leq z(x) \cdot (z - z(x)) = (z \cdot e_{n+1} - z(x)\cdot e_{n+1})(z(x) \cdot e_{n+1}) - \sum_{i=1}^n (z(x)\cdot e_i)^2.
    \end{equation}
    Since $z(x) \cdot e_{n+1} \geq 0$ (since $0\in E(K,x)$ would otherwise be closer to $z$) and $z(x)\neq z$ (since $z\notin E(K,x)$),  \eref{eqq} follows from \eref{eqq0}.
    Moreover, we have that 
    \begin{eqnarray}
\label{eqq1}
        (z - z(x))\cdot z(x) \leq ((z-z(x))\cdot e_{n+1})(z(x)\cdot e_{n+1}).
    \end{eqnarray}

    For $i=1,...,n$,  we now define the functions
    \begin{equation}
        f_i(x) := \frac{(z - z(x))\cdot e_i}{(z - z(x))\cdot e_{n+1}}, \quad x\in \Omega.
    \end{equation}
    Note that  $f_i\in C(\Omega)$ since the denominator never vanishes and $z(x)$ is a continuous function of $x$.  
    We next show that the linear operator $\cL:C(\Omega)\to \span\{f_1,\ldots,f_n\}$, given by
    $$
    \cL(f):=- \sum_{i=1}^n\lambda_i(f)f_i\,\,\in C(\Omega)
    $$
    is a good approximation to $f\in K$. First, for any $f\in K$ and $x\in \Omega$, we calculate
    \begin{eqnarray}
    \nonumber
        f(x) + \sum_{i=1}^n \lambda_i(f)f_i(x) &=& \frac{1}{(z - z(x))\cdot e_{n+1}}(z - z(x))\cdot \left(f(x)e_{n+1} + \sum_{i=1}^n \lambda_i(f)e_i\right)\\
        \nonumber
        &=& \frac{1}{(z - z(x))\cdot e_{n+1}}(z - z(x))\cdot E(f,x).
    \end{eqnarray}
    Since $E(f,x)\in E(K,x)$, it follows that $(z - z(x))\cdot (E(f,x) - z(x)) \leq 0$. Hence, we get
    \begin{equation}
        f(x) + \sum_{i=1}^n \lambda_i(f)f_i(x) \leq \frac{(z - z(x))\cdot z(x)}{(z - z(x))\cdot e_{n+1}}\leq z(x)\cdot e_{n+1} < z\cdot e_{n+1} = d^n(K)_{C(\Omega)} + \delta + \epsilon,
    \end{equation}
     where we have used \eref{eqq1} and \eref{eqq}.
    Applying the exact same argument to $-f\in K$, we get
    \begin{equation}
        \left|f(x) + \sum_{i=1}^n \lambda_i(f)f_i(x)\right| \leq d^n(K)_{C(\Omega)} + \delta + \epsilon, \quad x\in \Omega, \,
        f\in K.
    \end{equation}
    Hence,
    \begin{equation}
        \sup_{f\in K}\left\|f -  \cL(f) \right\|_{C(\Omega)} \leq d^n(K)_{C(\Omega)} + \delta + \epsilon.
    \end{equation}
    Since $\delta > 0$ and $\epsilon > 0$ were arbitrary, this implies $\delta_n(K)_{C(\Omega)} \leq d^n(K)_{C(\Omega)}$, as desired.
\hfill $\Box$

\subsection {Jackson and Bernstein inequalities for approximation spaces}
\label{A:JB}
In this section, we prove the following Jackson and Bernstein inequalities that were used in the proof of \eref{newAspaces}.  Let $s>s'>0$ and $X_0=\cA_{q'}^{s'} $ and $Y_0:=\cA^s_q$.
\vskip .1in
\noindent
{\bf Jackson  inequality}:   There is a constant $C>0$, depending only on $s,s',q,q'$ such that for each  $f\in Y_0$ we have 
\be
\label{JE1}
E_n(f)_{X_0}\le C(n+1)^{-(s-s')}\|f\|_{Y_0}, \quad n\ge 0.
\ee
\noindent
{\bf Proof:}
Since $\|\cdot\|_{\cA_\infty^s}\le C_0\|\cdot\|_{\cA_q^s}$, $0<q<\infty$ (see Lemma \ref{L:Aemb}),  it is enough to  prove this  when $Y_0=\cA_\infty^s$. 
Similarly, we only have to consider values of $q'$ that are {$0<q'<\infty$}. Let  $f\in U(\cA_\infty^s)$ and for each $k=0,1,\dots$, let $g_k$ be a best approximation to $f$ from 
$X_k$ in the $X$ quasi-norm. Then,  $\|f-g_k\|_X\le (k+1)^{-s} $, $k\ge 0$.  It follows that
\be
\label{approxfg1}
E_k(f-g_n)_X\le (n+1)^{-s},\quad  0\le k\le n,
\ee
and
\be
\label{approxfg2}
E_k(f-g_n)_X\le \|f-g_n-(g_k-g_n)\|_X =\|f-g_k\|_X \le (k+1)^{-s},\ k>n.
\ee
Hence, 
\be
\label{eachk}
\|f-g_n\|^{q'}_{X_0}   \le \sum_{k\ge 0}((k+1)^{s'q'-1} [\max\{k+1,n+1\}]^{-sq'}\le  C(n+1)^{-(s-s')q'},
\ee
as desired.  \hfill $\Box$

\vskip .1in
\noindent
{\bf Bernstein inequality}:   If $g\in X_n$, then
\be
\label{Bineq}
\|g\|_{Y_0}\le C \|g\|_{X_0} (n+1)^{s-s'},\quad n\ge 0,
\ee
where $C$   depend on $s,s'q,q'$.

\noindent
{\bf Proof:}  It is enough to prove this when $X_0=\cA_\infty ^{s'}$ and $q$ is small.  Let $g\in X_n$.  Then, we have $E_k(g)_X=0$, $k\ge n$  and $E_k(g)_X\le \|g\|_{X_0}(k+1)^{-s'}$,  $0\le k\le n$.  It follows that
\be 
\label{Bineq1}
\|g\|^q_{\cA_q^s}= \sum_{k=0}^{n} [(k+1)^sE_k(g)_X]^q(k+1)^{-1}\le 
\|g\|_{X_0}^q\sum_{k=0}^{n} (k+1)^{q(s-s')-1}\le C\|g\|_{X_0}^q(n+1)^{q(s-s')},
\ee
as desired.\hfill $\Box$

\vskip .5in

\noindent
Ronald DeVore, Department of Mathematics, Texas A\&M University, College Station, Texas 77843, ronalddevore@tamu.edu
\vskip .1in
\noindent
Guergana Petrova, Department of Mathematics, Texas A\&M University, College Station, Texas 77843, 
gpetrova@tamu.edu.
\vskip .1in
\noindent
Jonathan W. Siegel, Department of Mathematics, Texas A\&M University, College Station, Texas 77843, jwsiegel@tamu.edu
\vskip .1in
\noindent
Przemyslaw Wojtaszczyk,  Institute of Mathematics,
Polish Academy of Sciences, ul.
Sniadeckich 8, 00-656 Warsaw, Poland,
wojtaszczyk@icm.edu.pl

\vskip .1in
\begin{thebibliography}{99}

\bibitem{adams2003sobolev}
R. Adams, J. Fournier,
{Sobolev spaces},   Elsevier,  2003.


\bibitem{AK} F. Albiac, N. Kalton, {\em Topics in Banach space theory},  Graduate Texts in Mathematics, Springer, 2006.


\bibitem{babenko1991}
V. F. Babenko, {\it On best uniform approximations by splines in the presence of restrictions on their derivatives},
Mat.
Zametki {\bf 50}(6) (1991), 24–-30.

\bibitem{BS} C. Bennett, R. Sharpley, {\em Interpolation of Operators}, Academic Press, 1988.


\bibitem{BL} J Bergh, J. L{\"o}fstr{\"o}m, Interpolation Spaces: An Introduction,  Grundlehren der mathematischen Wissenschaften, 2011.


\bibitem{BCDDPW} P. Binev, A. Cohen, W. Dahmen, R. DeVore, G. Petrova, P. Wojtaszczyk, {\em Convergence rates for greedy algorithms in reduced bases methods}, SIAM J. Math. Anal. {\bf 43} (2011), 1457--1472.

\bibitem{BBDP} P. Binev,  A. Bonito, R. DeVore,  G. Petrova, {\em Optimal Learning}, Calcolo {\bf 61:15} (2024).

\bibitem{BMPPT}
A. Buffa, Y. Maday, A.T. Patera, C. Prud’homme,  G. Turinici, {\it A Priori convergence
of the greedy algorithm for the parameterized reduced basis}, M2AN Math. Model. Numer.
Anal., {\bf 46} (2012), 595–-603.

\bibitem{carl1981entropy}
B. Carl, {\it Entropy numbers, s-numbers, and eigenvalue problems},
Journal of Functional Analysis {\bf 41}(3) (1981), 290--306.

\bibitem{CS} B. Carl, I. Stephani, Entropy, compactness and the approximation of operators, Cambridge
University Press, 1990.


\bibitem{CDDmulti} A. Cohen, R. DeVore, W. Dahmen, {\em  Multiscale decompositions on bounded domains}, Transactions of the AMS,  {\bf 352}(8) (2000), 3651--3685.




\bibitem{CDPW} A. Cohen, R. DeVore, G. Petrova,  P.Wojtaszczyk, {\em  Optimal Stable Nonlinear Approximation}, J. FOCM,  {\bf 22} (2022), 607--647.


\bibitem{DHM} R. DeVore, R. Howard,  C. Micchelli,  {\em  Optimal Nonlinear Approximation}, Manuscripta mathematica,  {\bf 63} (1989), 
469--478.

\bibitem{DL} R. DeVore,    G.G. Lorentz,  {\em  Constructive Approximation}, Springer Grundlehren, Springer Verlag, 1993.

\bibitem{DPS} R. DeVore, G. Petrova,  J.W. Siegel, {\it Some remarks on interpolation spaces},  ArXiv ///




\bibitem{DPWg}
R. DeVore, G. Petrova, P. Wojtaszczyk, {\it Greedy Algorithms for Reduced
Bases in Banach Spaces}, Constr. Approx, {\bf 37}(3) (2013), 
455–-466.


\bibitem{DP} R. DeVore, V. Popov,  {\em   Interpolation of approximation spaces}, in "Constructive Theory of Functions 87," Publ. House of Bulg. Acad. Sci., Sofia, (1988), 110--119.

 \bibitem{DP1} R DeVore, V. Popov,  {\it Interpolation of Besov spaces}, Transactions AMS, {\bf 305}(1) (1988),  397--414.

\bibitem{DPA} R. DeVore, V. Popov, {\it Interpolation spaces and non-linear approximation} in ``Function Spaces and Applications'', (M. Cwikel, J. Peetre, Y. Sagher, and H. Wallin, Eds.) Springer Lecture Notes in Math., Vol. 1302, Springer, Berlin, (1988),  191--205.

\bibitem{DS} R. DeVore, R. Sharpley, {\it Besov spaces on domains in $\R^d$}, Transactions of AMS, {\bf 335}(2) (1993), 843--864.






\bibitem{GG} 
D. Garling,  Y. Gordon, {\em  Relations between some constants associated with  finite dimensional Banach spaces.} Israel J. Math. {\bf 9} (1971), 346--361.

\bibitem{gerhold2025entropyapproximationkolmogorovnumbers}
Marcus Gerhold, {\em Entropy-, Approximation- and Kolmogorov Numbers on Quasi-Banach Spaces.} arXiv preprint arXiv:2508.06542 (2025)


\bibitem{I} R.  Ismagilov, {\em Diameters of sets in normed linear spaces, and the approximation of functions by
trigonometric polynomials}, Uspekhi Mat. Nauk {\bf 29}(3)(177)  (1974), 161--178 (in Russian); Russian
Math. Surveys {\bf 29}(3) (1974), 169--186 (English translation).



\bibitem{KS} M. Kadec, M. Snobar, {\em Some functionals over a compact Minkowski space}, Math. Notes {\bf 10} (1971), 694--696.

\bibitem{K1} B. Kashin, 
{\it On Kolmogorov diameters of octahedra}, Dokl. Akad. Nauk SSSR {\bf 214} (1974),
1024--1026; English transl. in Soviet Math. Dokl. {\bf 15} (1974).

\bibitem{K2} B. Kashin,  {\it Diameters of some finite-dimensional sets and classes of smooth functions},
Math. USSR, Izv., {\bf 11} (1977), 317--333.

\bibitem{Kashin}
B. Kashin, {\it The widths of certain finite dimensional sets and classes of smooth functions},
Izv. {\bf 41} (1977), 334--351.


\bibitem{Kol} A. Kolmogoroff, \"Uber die beste Ann  \"aherung von Funktionen einer gegebenen Funktionenklasse, Ann.
Math., {\bf 37} (1936), 107-–111.


\bibitem{konovalov1984}
V. N. Konovalov, {\it Estimates of Kolmogorov-type widths for classes of differentiable periodic functions},
Mat.
Zametki {\bf 35}(3) (1984), 369–-380.

\bibitem{konovalov2002}
V. N. Konovalov, {\it Approximation of Sobolev classes by their finite-dimensional sections},
Mat.
Zametki {\bf 72}(3) (2002), 370–-382.


\bibitem{LS} 
D. Leviatan, I. Shevchuk,  
{\em Comparing the degrees of unconstrained and shape preserving approximation by polynomials}, J. Approx. theory,  
 {\bf 211}(2) (2016), 16--28.

\bibitem{LP} J. Lindenstrauss, A. Pełczyński, {\em Absolutely summing operators in $\cL_p$ spaces and their applications,} Studia Math, {\bf 29} (1968), 275-326



\bibitem{LGM} {G.G. Lorentz, M.Golitschek, Y.Makovoz}, {\em Constructive Approximation, Advanced Problems, vol II} ,Grundlehren der mathematischen Wissenschaften 304, Springer Verlag, 
 1996.




\bibitem{malykhin2016relative}
Y. Malykhin, {\it Relative widths of Sobolev classes in the uniform and integral metrics},
Proceedings of the Steklov Institute of Mathematics {\bf 293}(1)  (2016), 209--215.

 
 
\bibitem{OO} T. Oikhberg, M. Ostrovskii, {\em Dependence of Kolmogorov Widths on the Ambient
Space}, Journal of Mathematical Physics, Analysis, Geometry {\bf 9} (2013),
25--50.

\bibitem{O} T. Oikhberg, {\em Absolute widths of some embeddings}, J. Approx. Theory, {\bf 81}(1) (1995), 120--126.


\bibitem{PW} Guergana Petrova, Przemek Wojtaszczyk,  {\it Lipeschitz widths}, Constr. Approximation, {\bf 7} (2023), 759--805.

\bibitem{Pinkus} Allan Pinkus, {\em $n$-Widths in Approximation Theory}, Springer, 1985.

\bibitem{stein1970singular}
Elias M. Stein, {\it Singular integrals and differentiability properties of functions},
Princeton University Press, No.\ 30, 1970.

\bibitem{telyakovskii2001}
Yu. N. Subbotin, S. A. Telyakovskii, {\it Relative widths of classes of differentiable functions in the $L^2$ metric},
Usp. Mat. Nauk {\bf 56}(4) (2001), 159–-160.

\bibitem{T} V. Tikhomirov, {\em Widths of sets in function spaces and the theory of best approximations}, Uspekhi Mat. Nauk, 15(3), (93), (1960), 81--120. English translation in Russian Math. Survey, 15, 1960.


\bibitem {hinrichs2016carl} Aicke Hinrichs, Anton Kolleck, and Jan Vybiral, {\it Carl's inequality for quasi-Banach spaces}, Journal of Functional Analysis {\bf 271}-8 (2016), 2293--2307.

\bibitem{triebel2006theory}
Hans Triebel, {\it Theory of Function Spaces},
Birkh\"auser Basel, 1983.

\bibitem {P} P. Wojtaszczyk, {\it On greedy algorithm approximating Kolmogorov widths in Banach spaces}, J. Math. Anal. Appl., {\bf 424} (2015), 685--695. 

\end{thebibliography}
\end{document}